\documentclass{article}
\usepackage[utf8]{inputenc}
\usepackage[T1]{fontenc}
\usepackage[rightcaption]{sidecap}

\usepackage[inner=30mm,outer=30mm,tmargin=30mm,bmargin=30mm]{geometry}

\usepackage{latexsym,amsfonts,amsmath,amssymb,epsfig,tabularx,amsthm,dsfont,mathrsfs}
\usepackage{graphicx}
\usepackage{enumerate}
\usepackage{bbm} 
\usepackage{tikz}
\usepackage{verbatim}
\usetikzlibrary{arrows}
\usetikzlibrary{plotmarks}

\usepackage[colorlinks=false]{hyperref} 
\hypersetup{linkbordercolor=red,
        linkcolor=black,
        citecolor=black,
        urlcolor=blue}
\usepackage{caption}

\usepackage{subfig}
				
\usepackage[framed,numbered,autolinebreaks,useliterate]{mcode}
\usepackage[]{listings} 
\usepackage{cprotect} 

\newcommand{\ii}{\mathrm{i}}
\newcommand{\e}{\mbox{e}}

\newcommand{\eim}[1]{{\e}^{-2\pi{\ii} #1}}
\newcommand{\T}{\mathbb{T}}

\newcommand{\md}{\mathup{md}_{\T}}
\newcommand{\sep}{\mathup{sep}}

\newcommand{\Lip}{\mathup{Lip}}

\DeclareMathOperator{\supp}{supp}

\DeclareMathOperator*{\argmin}{argmin}

\DeclareMathOperator{\dist}{dist}

\DeclareMathOperator{\Id}{Id}

\DeclareMathOperator{\proj}{proj}

\newcommand{\R}{{\mathbb R}}
\newcommand{\C}{{\mathbb C}}
\newcommand{\N}{{\mathbb N}}
\newcommand{\Z}{{\mathbb Z}}

\DeclareMathAlphabet{\mathup}{OT1}{\familydefault}{m}{n}

\newcommand{\widebar}[1]{\mbox{\kern1.5pt\hbox{\vbox{\hrule height 0.6pt \kern0.35ex
        \hbox{\kern-0.15em \ensuremath{#1 }\kern0.0em}}}}\kern-0.1pt}
\newcommand{\ESPRIT}{\mathup{ESPRIT}}

\newlength{\fixboxwidth}
\usepackage{color}
\definecolor{darkgreen}{rgb}{0,0.5,0}

\usepackage{algorithmic}
\usepackage[Algorithm]{algorithm}
\usepackage[rightcaption]{sidecap}

\usepackage{selinput}
\SelectInputMappings{
  adieresis={ä},
  germandbls={ß}
}

\usepackage[english]{babel}
\usepackage[capitalise]{cleveref}
\theoremstyle{plain}

\newtheorem{Thm}{Theorem}[section]
\newtheorem{Lem}[Thm]{Lemma}

\newtheorem{Cor}[Thm]{Corollary}

\theoremstyle{definition}
\newtheorem{Rem}[Thm]{Remark}

\newtheorem{Def}[Thm]{Definition}

\newtheorem*{Acknowledgements}{Acknowledgements}

\crefname{subsection}{Subsection}{Subsections}

\crefname{Lem}{Lemma}{Lemmas}
\Crefname{Lem}{Lemma}{Lemmas}
\crefname{Thm}{Theorem}{Theorems}
\Crefname{Thm}{Theorem}{Theorems}
\crefname{Prop}{Proposition}{Propositions}
\Crefname{Prop}{Proposition}{Propositions}
\crefname{Cor}{Corollary}{Corollaries}
\Crefname{Cor}{Corollary}{Corollaries}
\crefname{Def}{Definition}{Definitions}
\Crefname{Def}{Definition}{Definitions}
\crefname{Rem}{Remark}{Remarks}
\Crefname{Rem}{Remark}{Remarks}

\title{Sparse super resolution is Lipschitz continuous}
\author{Mathias Hockmann\thanks{Osnabr\"uck University, Institute of Mathematics \texttt{\{mahockmann,skunis\}@uos.de}} \thanks{Osnabr\"uck University, Research Center of Cellular Nanoanalytics} \and Stefan Kunis\footnotemark[1] \footnotemark[2]}

\begin{document}
\bibliographystyle{abbrv}
\maketitle

\section*{Abstract}
Motivated by the application of neural networks in super resolution microscopy, this paper considers super resolution as the mapping of trigonometric moments of a discrete measure on $[0,1)^d$ to its support and weights. We prove that this map satisfies a local Lipschitz property where we give explicit estimates for the Lipschitz constant depending on the dimension $d$ and the sampling effort. Moreover, this local Lipschitz estimate allows to conclude that super resolution with the Wasserstein distance as the metric on the parameter space is even globally Lipschitz continuous. As a byproduct, we improve an estimate for the smallest singular value of  multivariate Vandermonde matrices having pairwise clustering nodes.

\medskip
\noindent\textit{Key words and phrases}:
	frequency analysis,
	Wasserstein metric,
	super resolution,
	stability analysis.
	
	\medskip
	
	\noindent\textit{2020 AMS Mathematics Subject Classification} : \text{
		65T40, 
		42B05, 
		49Q22. 
	}

\section{Introduction} 
In recent papers \cite{Nehme_20,Nehme_18,Speiser_20}, methods for single molecule localisation microscopy by deep neural networks have been proposed, but a theoretical guarantee for the quality of neural network approaches is still missing.
We follow the common approach to distinguish expressivity, training, and generalisation of neural networks and our sole interest is on the first issue.

Of course, a deeper understanding of the particular underlying inverse problem is needed and therefore we study single molecule localisation microscopy formulated mathematically as the task to recover the pairwise different \emph{nodes} $t_j\in[0,1)^d$ and the \textit{weights} $c_j\in\C$,  $j=1,\dots,M$, (referred to as parameters) of the Dirac measure
\begin{align} \label{Dirac_sum}
\mu(x)=\sum_{j=1}^M c_j \delta_{t_j}(x)
\end{align}
from measurements. A frequently used assumption is that the measurements are given through the \textit{trigonometric moments}
\begin{align} \label{eq_spectral_moments}
    \hat{\mu}(k)=\int_{[0,1]^d } \eim{kx} d\mu(x) = \sum_{j=1}^M c_j e^{-2\pi i k t_j}, \quad k\in \mathcal{B},
\end{align}
of the measure $\mu$, where $\mathcal{B}\subset \Z^d$ is some finite set.
The measurement process is also referred to as the \textit{moment map} and we call the inverse mapping the \textit{Prony-map}.
For the univariate case $d=1$, one can think of the index set of known moments being
$\mathcal{B}=\left\{-N,-N+1,\dots,N-1,N\right\}$ for some $N\in\N$ and Prony's method \cite{Prony_1795} reconstructs the parameters if $N\geq M$.
Generalisations for $d>1$ have been proposed e.g.~in \cite{KuPeRoOh16}
but suffer from a somewhat more complicated algorithmic framework.

Here, we are interested in whether it is possible that a neural network approximates the Prony-map.
In contrast to the approach of recreating one of the standard Prony-like methods via a neural network, we analyse the regularity of the Prony-map in order to apply results about the expressivity of neural networks (cf.\,\cite{Elbraechter_19,Guehring_20,Yarotsky_16,Zhou_20}).
We aim at a quantitative description of the features of the Prony-map and prove an upper bound for its Lipschitz constant.\footnote{If we use $2M$ moments, we see directly by \eqref{eq_spectral_moments} that the moments are entire functions of the $M$ nodes and $M$ weights such that the inverse mapping theorem gives local smoothness of the Prony-map around every point where the Jacobian of the mapping from the moments to the parameters has non-vanishing determinant. As the determinant is analytic, we have local smoothness of the Prony-map almost everywhere. Nevertheless, we are interested in more applicable and quantitative statements about the regularity of the process of super resolution.}

Our main result is \cref{Cor_global_Lipschitz}, stating that complex measures $\mu_1,\mu_2$ each having well-separated support nodes satisfy
\begin{align*}
    W_1(\mu_1,\mu_2)
    \leq  2.3 \cdot\|\hat{\mu}_1-\hat{\mu}_2\|_2
\end{align*}
where $W_1$ is the 1-Wasserstein distance and the norm on the right hand side is on the low-pass region $\mathcal{B}=\{k\in\Z^d: \|k\|_2\leq N\}$. This builts upon local Lipschitz results which have been developed for the univariate and a particular and different bivariate case first by Diederichs in \cite{Diederichs_19,Diederichs_18}.
We emphasise that these local results generalise some Lipschitz-like estimates known from the stability analysis of subspace methods like the matrix pencil method \cite{Hua_1990} or the ESPRIT algorithm \cite{Roy_1989,Potts_17}.
Technically, all approaches heavily rely on the construction of \textit{localising functions} (extremal functions or minorants) as introduced by Beurling and Selberg (cf.\,\cite{Vaaler_1985}) in the early 20th century.
Such functions also have been used to establish lower bounds of the form
\begin{align*}
    \sum_{k\in\mathcal{B}} \left|\sum_{j=1}^M c_j \eim{k t_j}\right|^2 \geq C \|c\|_2^2
\end{align*}
for well-separated nodes $t_j$ and some constant $C=C(\mathcal{B},\{t_j\}_{j=1}^M)>0$ by Moitra, Aubel, and Bölcskei \cite{Moitra_15,Aubel_19} and by the second author and others \cite{Kunis_17,KuNaSt21} for the univariate and multivariate case, respectively.
As a byproduct of our main result here, we improve upon recent estimates for the above constant $C$ in the case of pairwise clustering nodes \cite{Nagel_20}.

This paper is organized as follows: In \cref{sec_Preliminaries}, we rigorously define the Prony-map as the inverse of the mapping from the parameters to the moments. The main results about the local and global Lipschitz property are explained in \cref{sec_local_Lipschitz} where we proceed from the univariate case to the bivariate case and later to arbitrary dimension $d$. Finally, the results for the smallest singular value of Vandermonde matrices with pairwise clustering nodes are given in \cref{sec_Vandermonde}.

\section{Preliminaries and known stability of subspace methods} \label{sec_Preliminaries}
We define the following quantities on the $d$-dimensional torus $\T^d=\left(\R/\Z\right)^d\cong [0,1)^d$ which is the domain where the nodes of the measure come from. At first, we present the notation for the univariate case $d=1$ and generalize it for higher dimensions in \cref{subsec_bivariate}.
\begin{Def}[Separation and matching distance] \label{Def_matching_distance}
For a set $Y=\{t_1,\dots,t_M\}\subset\T$ we define the \emph{minimal-separation distance}
\begin{align*}
    \sep\,Y:=\min_{j\neq l} \|t_j-t_l\|_{\T}:=\min_{r\in\Z, j\neq l} |t_j-t_l+r|.
\end{align*}
The difference to a second set $Y'=\left\{t_1',\dots,t_M'\right\}\subset \T$ of equal cardinality is defined by the \textit{matching distance}
\begin{align*}
    \md(Y,Y'):=\min_{\pi} \max_{j\in\{1,\dots,M\}} \|t_j-t_{\pi(j)}'\|_{\T}
\end{align*}
where $\pi$ can be an arbitrary permutation on $\{1,\dots,M\}$. See \cref{fig:sep_on_torus} for an illustration.
\end{Def}
\begin{figure}[ht]
    \centering
    \begin{tikzpicture}[>=triangle 45,x=10cm,y=0.2cm,scale=1]
    \draw[-] (0,0) -- (1,0);
    \draw[-] (0,-0.1) node[below]{0} -- (0,0.1); 
    \draw[-] (1,-0.1) node[below]{1} -- (1,0.1); 
    \foreach \x in {0.15,0.2,0.6,0.9}
    \draw[mark=*,mark size=3pt,mark options={color=green}] plot coordinates {(\x,0)};
    \draw[mark=*,mark size=3pt,mark options={color=green}] plot coordinates {(1.1,0.5)} node[right]{\,$Y$};
    \foreach \x in {0.12,0.3,0.65,0.75}
    \draw[mark=triangle*,mark size=3pt,mark options={color= blue}] plot coordinates {(\x,0)};
    \draw[mark=triangle*,mark size=3pt,mark options={color=blue}] plot coordinates {(1.1,-1)} node[right]{\,$Y'$};
    \draw[|-|] (0.15,-1) -- node[below]{$\sep\,Y$} (0.2,-1);
    \draw[|-|] (0.65,-1) -- node[below]{$\sep\,Y'$} (0.75,-1);
    \draw[|-|] (0.75,1) -- node[above]{$\md(Y,Y')$} (0.9,1);
    \end{tikzpicture}
    \caption{Definition of separation distance on $\T$ and matching distance between two finite sets $Y,Y'\subset \T$. The matching distance might be seen as induced by the maximum norm on $\T^M$ modulo the symmetric group.}
    \label{fig:sep_on_torus}
\end{figure}
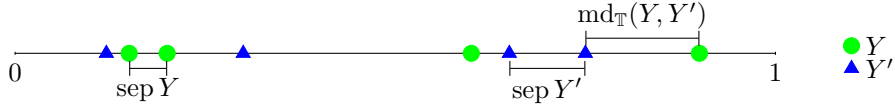

\begin{Def}[Probability-like atomic measures and their trigonometric moments]
We call a complex Borel measure $\mu$ on a topological space $X$ a \emph{probability-like measure} if $\mu$ has normalized mass, i.e.\,$\mu(X)=1$, and denote the set of probability-like measures on $\T$ by $\mathcal{M}(\T)$.\footnote{We restrict ourselves to these kind of measures in order to have the Wasserstein distance as a metric on $\mathcal{M}(\T)$.} For some specified minimal weight $c_{\min}>0$ and minimal separation $q>0$, we consider the set of measures
\begin{align*}
    \mathcal{M}_{c_{\min}}(q) := \left\{\sum_{t\in Y} c_t \delta_{t}: \quad c_t\in\C,\,|c_t|\geq c_{\min},\,\sum_{t\in Y} c_t =1,\,Y\subset \T,\,\sep\, Y\geq q\right\} \subset \mathcal{M}(\T)
\end{align*}
and we denote the \textit{node set} of $\mu\in \mathcal{M}_{c_{\min}}(q)$ by $Y^{\hat{\mu}}$. For $N\in\N$, the truncated moment map now assigns to each measure $\mu\in \mathcal{M}_{c_{\min}}(q)$ its trigonometric moments $\hat \mu\in\C^{2N+1}$ and one can also think about this as the Fourier transform of $\mu$ evaluated at frequencies being the integers between $-N$ and $N$. Hence, the image of this map is the following set of exponential sums, see also \cite[section 3]{Diederichs_19},
\begin{align*}
    \widehat{\mathcal{M}}_{c_{\min}}^N(q) := \left\{\left(\sum_{t\in Y} c_t \eim{t k}\right)_{k=-N}^N:\quad \sum_{t\in Y} c_t \delta_{t}\in \mathcal{M}_{c_{\min}}(q)\right\} \subset \C^{2N+1}.
\end{align*}

\end{Def}

The standard norm on the vector space $\C^{2N+1}$,
\begin{align*}
    \|\hat{\mu}\|_2=\left(\sum_{k=-N}^N |\hat{\mu}(k)|^2\right)^{1/2},
\end{align*}
always induces a metric on $\widehat{\mathcal{M}}_{c_{\min}}^N(q)$.
We will see later in \cref{Thm_Lipschitz_univariate} that the previously mentioned truncated moment map is injective  for $q>\frac{2}{N}$ such that the following inverse map is well defined.

\begin{Def}[Prony-map]
Let $N\in\N$ and $c_{\min}>0$. Then, the \textit{Prony-map} $\mathscr{P}$ is the map
\begin{align*}
    \mathscr{P}: \quad \bigcup_{q>\frac{2}{N}} \widehat{\mathcal{M}}_{c_{\min}}^N(q) \to \mathcal{M}(\T), \quad \left(\sum_{t\in Y} c_t \eim{t k}\right)_{k=-N}^N \mapsto \sum_{t\in Y} c_t \delta_t.
\end{align*}

\end{Def}

We want to understand the stability of the inverse problem of reconstructing the measure $\mu$ from its trigonometric moments. This means that we are interested in the difference between the two measures which are the outputs of some optimal algorithm 
if we use two similar moment vectors as inputs for the algorithm. By taking known algorithms like the parametric approaches of matrix pencil (MP) or estimation of signal parameters via rotational invariance techniques (ESPRIT) and their stability results, we can already get an impression what is at least possible for an optimal solution of the inverse problem. Especially, the stability of ESPRIT has been extensively studied recently (cf.\,\cite{Li_20,Nagel_20,Potts_17}):

\begin{Def}[ESPRIT-map] \label{Def_ESPRIT_map}
Let the number of nodes $M$ be known for the ESPRIT-method. We define $\mathscr{P}_{\ESPRIT}: \C^{2N+1} \to \T^M$ mapping any perturbed moment vector to the vector of nodes computed by the ESPRIT-method. We call this mapping \textit{ESPRIT-map}.
\end{Def}
Note that we simplify the situation by just considering the mapping of the moments to the nodes and not to the weights for the ESPRIT algorithm. This is already a major issue with parametric methods where estimation of the parameters is usually divided into the computation of the nodes and weights successively. For two finite subsets of $\T$ with the same number of elements, one might use the matching distance displayed in \cref{fig:sep_on_torus} and defined in \cref{Def_matching_distance}.
If the nodes are separated by $\frac{2}{N+1}$, the number of nodes $M$ is known and the error in the moments is small, one can apply the following result by Nagel (cf.\,\cite{Nagel_20}) which is based on \cite{Aubel_16,Li_20,Potts_17} and presented with a shift in the frequencies and a different assumption on the separation in order to fit into our setting.
\begin{Thm} \textup{(Stability of ESPRIT, $d=1$, cf.\,\cite[Thm 4.3.14]{Nagel_20})} \label{Prop_Dominik}
For $\mu_0\in \mathcal{M}_{c_{\min}}(\frac{2}{N+1})$, 
the truncated moment vector is denoted by $\hat{\mu_0}\in \C^{2N+1}$ and the ESPRIT-method has access to the true number of nodes $M$ as well as to a perturbed moment vector $\hat{\mu}\in\C^{2N+1}$. If the perturbation satisfies
\begin{align*}
    \|\hat{\mu}-{\hat\mu_0}\|_{\infty} < \frac{c_{\min}}{60},
\end{align*}
then the matching distance between the reconstructed nodes $\mathscr{P}_\ESPRIT({\hat{\mu}})$ and the ground truth is
\begin{align}  \label{eq_Dominik}
    \md(\mathscr{P}_\ESPRIT(\hat{\mu}_0),\mathscr{P}_\ESPRIT(\hat{\mu}))\leq \frac{190 M}{c_{\min}} \|\hat{\mu}-{\hat\mu_0}\|_{\infty}.
\end{align}
\end{Thm}

\begin{Rem}
In the language of \cref{Def_ESPRIT_map} we can formulate this as follows: For every ground truth $\hat{\mu}_0\in \widehat{\mathcal{M}}_{c_{\min}}^N(\frac{2}{N+1})$ there are $\delta>0$ and $L>0$ such that 
\begin{align} \label{eq_Lipschitz1}
    \md(\mathscr{P}_\ESPRIT(\hat{\mu}_0),\mathscr{P}_\ESPRIT(\hat{\mu}))\leq L \|\hat{\mu}-\hat{\mu}\|_{\infty}
\end{align}
for all $\hat{\mu}\in B_{\delta}(\hat{\mu}_0)$. Note that $\delta=\delta(c_{\min})$ and $L=L(c_{\min},|Y^{\hat{\mu}_0}|)$ depend on $\hat{\mu}_0\in \widehat{\mathcal{M}}_{c_{\min}}^N(\frac{1}{N+1})$ by \cref{Prop_Dominik}. Hence, we can not prove by this that $\mathscr{P}_\ESPRIT$ is locally Lipschitz continuous in the usual way since the Lipschitz condition \eqref{eq_Lipschitz1} does not hold for arbitrary pairs $\hat{\mu}_1,\hat{\mu}_2\in B_{\delta}(\hat{\mu}_0)$ but only for the matching distance between  $\mathscr{P}_\ESPRIT(\hat{\mu})$ and the ground truth $\mathscr{P}_\ESPRIT(\hat{\mu}_0)$.\footnote{Local Lipschitz continuity would demand that $\md(\mathscr{P}_\ESPRIT(\hat{\mu}_1),\mathscr{P}_\ESPRIT(\hat{\mu}_2))$ can be bounded in terms of $\|\hat{\mu}_1-\hat{\mu}_2\|_{\infty}$ for every pair $\hat{\mu}_1,\hat{\mu}_2\in B_{\delta}(\hat{\mu})$. However, applying the triangle inequality and \eqref{eq_Dominik} gives a bound of $\md(\mathscr{P}_\ESPRIT(\hat{\mu}_1),\mathscr{P}_\ESPRIT(\hat{\mu}_2))$ in terms of $\|\hat{\mu}_0-\hat{\mu}_1\|_{\infty}+\|\hat{\mu}_0-\hat{\mu}_2\|_{\infty}$.}
\end{Rem}

Compared to the ESPRIT-map, the order in $N$ in the stability result \eqref{eq_Dominik} might be improved for the Prony-map:
\begin{Rem} \label{Rem_Batenkov_2020}
Assume as in \cite{Batenkov_2020} that the node set of $\hat{\mu}$, $Y^{\hat{\mu}}=\{t_j\}_j$, does only have one cluster $Y\subset Y^{\hat{\mu}}$ where the nodes can lie very closely together, whereas all nodes  outside of the cluster, i.e.\,$t_\ell\in Y^{\hat{\mu}}\setminus Y$, satisfy
\begin{align*}
    |t_\ell-t_j|\geq \frac{c}{N}
\end{align*}
for all $t_j\in Y^{\hat{\mu}}$, $j\neq\ell$, and a constant $c>0$. In contrast to our setting, \cite{Batenkov_2020} starts from a more information theory based approach of assuming continuous knowledge about the Fourier transform of the measure $\mu$ instead of discrete moments as in \eqref{eq_spectral_moments}. The main result \cite[Thm.\,2.8]{Batenkov_2020} compares the nodes $t_j^{(1)}$ and $t_j^{(2)}$ corresponding to $\hat{\mu}_1$ and $\hat{\mu}_2$ lying in a $\epsilon$-neighbourhood of $\hat{\mu}$ for small enough $\epsilon>0$. It shows that the stability of the recovery problem for the $\ell$th node $t_\ell$ can be formulated as
\begin{align}
    |t_\ell^{(1)}-t_\ell^{(2)}| \asymp \frac{\|\hat{\mu}_1-\hat{\mu}_2\|_{\infty}}{N} \label{eq_Batenkov_2020}
\end{align}
if $t_\ell\in Y^{\hat{\mu}}$ does not belong to the cluster, i.e.\,$t_\ell\notin Y$. The constants hidden behind the $\asymp$ symbol\footnote{We write $a\asymp b$ if there are positive constants $c_1,c_2$ such that $c_1b\leq a\leq c_2 b$.} are said to depend on the number of nodes, a-priori bounds for the weight parameters $|c_j|$ and geometric parameters of the node set but not on $N$. Even if this gives only a local Lipschitz result in a slightly different setup, these results indicate that one might prove a global Lipschitz property of the Prony-map which we do in the next section. We also make the constant $c$ for the separation of the nodes as well as the Lipschitz constant in estimates like \eqref{eq_Batenkov_2020} explicit.   
\end{Rem}

\section{Main results} \label{sec_local_Lipschitz}
At first, we review known local results for the uni- and bivariate case. After that, we present our own approach to obtain local and global results for arbitrary dimensions $d$. 
\subsection{Local Lipschitz property in the univariate case}
The findings of Diederichs \cite{Diederichs_19,Diederichs_18} improving Moitra's lower bound can be applied in order to establish a local Lipschitz property of the Prony-map for $d=1$.
\begin{Thm}\textup{(Diederichs, cf.\,\cite[Cor.\,5.1]{Diederichs_19})} \label{Prop_lower_bound}
Let $\hat{\mu}_1,\hat{\mu}_2\in \widehat{\mathcal{M}}_{c_{\min}}^N(\frac{3}{N+1})$. Assume that the difference of $\hat{\mu}_1$ and $\hat{\mu}_2$ satisfies
    \begin{align}
        \|\hat{\mu}_1-\hat{\mu}_2\|_2^2 = \sum_{k=-N}^N |\hat{\mu}_1(k)-\hat{\mu}_2(k)|^2 < \frac{4N+4}{3} c_{\min}^2. \label{eq_Diederichs_condition}
    \end{align}
Then, we have $|Y^{\hat{\mu}_1}|=|Y^{\hat{\mu}_2}|$ and for every $t\in Y^{\hat{\mu}_1}$ there is a unique $t'=:\eta(t)\in Y^{\hat{\mu}_2}$ with $\|t-\eta(t)\|_{\T}<\frac{3}{2N+2}$. Additionally, the estimate
    \begin{align}
        \|\hat{\mu}_1-\hat{\mu}_2\|_2^2 &\geq \frac{2\pi^2 (N+1)^3}{3^5} \sum_{t\in Y^{\hat{\mu}_1}} \left(|c_t|^2+|c_{\eta(t)}|^2\right) \|t_j-\eta(t_j)\|_{\T}^2 + \frac{N+1}{3} \sum_{t\in Y^{\hat{\mu}_1}} |c_t-c_{\eta(t)}|^2 \nonumber\\
        &\geq \frac{4\pi^2 (N+1)^3}{3^5} c_{\min}^2 \sum_{t\in Y^{\hat{\mu}_1}} \|t-\eta(t)\|_\T^2 + \frac{N+1}{3} \sum_{t\in Y^{\hat{\mu}_1}} |c_t^{(1)}-c_{\eta(t)}^{(2)}|^2 \label{eq_Diederichs}
    \end{align}
    holds.
\end{Thm}
The methods which we will develop for the case of $d>1$ later and which are also based on Diederichs approach are applicable to the univariate situation as well. We obtain the following theorem with the theoretical background of \cref{subsec_bivariate}.
\begin{Thm}[Univariate Lipschitz] \label{Thm_Lipschitz_univariate}
Let $\hat{\mu}_1,\hat{\mu}_2\in \widehat{\mathcal{M}}_{c_{\min}}^N(\frac{2\kappa}{N})$ for some $\kappa>1$ and assume that
\begin{align}
    \sum_{k=-N}^N |\hat{\mu}_1(k)-\hat{\mu}_2(k)|^2 < \frac{3\kappa^2-1}{2\kappa^3} N c_{\min}^2. \label{eq_cond_univ}
\end{align}
Then, we have $|Y^{\hat{\mu}_1}|=|Y^{\hat{\mu}_2}|$ and for every $t\in Y^{\hat{\mu}_1}$ there is a unique $t'=:\eta(t)\in Y^{\hat{\mu}_2}$ with $\|t-\eta(t)\|_{\T}<\frac{\kappa}{N}$. Furthermore, the difference in the moments is related to the difference in the nodes and weights via 
\begin{align*}
    \sum_{k=-N}^N |\hat{\mu}_1(k)-\hat{\mu}_2(k)|^2 \geq \frac{15}{4} \frac{\kappa^2-1}{\kappa^5}N^3 c_{\min}^2  \sum_{t\in Y^{\hat{\mu}_1}} \|t-\eta(t)\|_\T^2 + \frac{1}{4} \frac{3\kappa^2-1}{2\kappa^3} N \sum_{t\in Y^{\hat{\mu}_1}} \left|c_t^{(1)}-c_{\eta(t)}^{(2)}\right|^2.
\end{align*}
For $\kappa^2\geq \frac{13}{9}$, one can improve the statement and we have that $\|t-\eta(t)\|_{\T}<\frac{\kappa}{2N}$ for all $t\in Y^{\hat{\mu}_1}$ as well as
\begin{align}
    \sum_{k=-N}^N |\hat{\mu}_1(k)-\hat{\mu}_2(k)|^2 \geq 10 \frac{\kappa^2-1}{\kappa^5} N^3 c_{\min}^2 \sum_{t\in Y^{\hat{\mu}_1}} \|t-\eta(t)\|_\T^2 + \frac{\kappa^2+1}{4\kappa^3} N \sum_{t\in Y^{\hat{\mu}_1}} |c_t^{(1)}-c_{\eta(t)}^{(2)}|^2 \label{eq_Lip_uni}
\end{align}
where the term with the difference of the nodes becomes maximal in $\kappa$ for $\kappa=\sqrt{\frac{5}{3}}\approx 1.291$.
\end{Thm}
\begin{proof}
See appendix.
\end{proof}
\begin{Rem} \label{Rem_Diederichs_1D}
\begin{itemize}
    \item [1)] Diederichs (cf.\,\cite[p.\,11]{Diederichs_19}) remarks that the minimal separation can be decreased to the condition $\frac{2}{N+1}$ yielding worse rates in $N$ for the estimate \eqref{eq_Diederichs}. In the proof of \cref{Thm_Lipschitz_univariate}, we find that the same worse rate in $N$ can be observed by our construction if $\kappa=1$. Moreover, Diederichs gives an example that a similar result with a lower bound for the separation of the node sets of $\hat{\mu}_1,\hat{\mu}_2$ strictly smaller than $\frac{2}{N+1}$ is not possible.
    \item [2)] We compare the conditions and results of \cref{Prop_lower_bound} and \cref{Thm_Lipschitz_univariate} in \cref{table_univariate_comparison} and observe that \cref{Thm_Lipschitz_univariate} has weaker conditions of the separation of the measure but stronger conditions on the difference of the moment vectors $\hat{\mu}_1$ and $\hat{\mu}_2$. Taking the optimal $\kappa$ for the constant appearing in the term with the difference of the nodes, gives a result which is better by a factor roughly ten for this constant and almost the same constant in front of $\sum_{t\in Y^{\hat{\mu}_1}} |c_t^{(1)}-c_{\eta(t)}^{(2)}|^2$. All terms have the same asymptotic dependency on $N$ and $c_{\min}$.
    \begin{table}[ht]
        \centering
        \begin{tabular}{m{2cm}|p{2.5cm}|p{1.95cm}|p{3.1cm}|p{2.5cm}}
             & Separation & $\|\hat{\mu}_1-\hat{\mu}_2\|_2^2<$ & 
             \multicolumn{2}{c}{Constant in front of} \\
             & of $\mu_1$ and $\mu_2$ & & $\sum\|t-\eta(t)\|_\T^2$ & $\sum |c_t^{(1)}-c_{\eta(t)}^{(2)}|^2$\\\hline
             \rule{0pt}{0.6cm} \cref{Prop_lower_bound} (Diederichs) & $\geq \frac{3}{N+1}$ & $\frac{4}{3} (N+1) c_{\min}^2$  & \parbox[]{3cm}{$\frac{4\pi^2}{3^5} (N+1)^3 c_{\min}^2$ \\ $\approx 0.16 \cdot (N+1)^3 c_{\min}^2$} & $\frac{N+1}{3}$ \\\hline
             \rule{0pt}{0.6cm} \cref{Thm_Lipschitz_univariate} \vspace{0.005cm}& \parbox[]{1.6cm}{$\geq\frac{2\kappa}{N}$ with \\ $\kappa>1$} & $\frac{3\kappa^2-1}{2\kappa^3} N c_{\min}^2$ & $\frac{15}{4} \frac{\kappa^2-1}{\kappa^5}N^3 c_{\min}^2$ & $\frac{1}{4} \frac{3\kappa^2-1}{2\kappa^3} N $ \\\hline
             \rule{0pt}{1.2cm} & \parbox[]{2.5cm}{$\geq\frac{2\kappa}{N}$ with \\ $\kappa\geq\sqrt{\frac{13}{9}}\approx 1.20$} & $\frac{3\kappa^2-1}{2\kappa^3} N c_{\min}^2$ & $10\frac{\kappa^2-1}{\kappa^5}N^3 c_{\min}^2$ & $\frac{\kappa^2+1}{4\kappa^3}  N $ \\\hline
             \rule{0pt}{1.2cm} & \parbox[]{2.5cm}{$\geq\frac{2\kappa}{N}$ with \\ $\kappa=\sqrt{\frac{5}{3}}\approx 1.29$} & $\frac{6}{5} \sqrt{\frac{3}{5}} N c_{\min}^2$ & \parbox[]{3cm}{$\frac{12}{5}\sqrt{\frac{3}{5}} N^3 c_{\min}^2$ \\ $\approx 1.86\cdot N^3 c_{\min}^2$}  & \parbox[]{2.4cm}{$\frac{2}{5}\sqrt{\frac{3}{5}}  N $ \\ $\approx 0.31\cdot N$} \\\hline
        \end{tabular}
        \caption{Comparison of the constants appearing in \cref{Prop_lower_bound} and \cref{Thm_Lipschitz_univariate}.}
        \label{table_univariate_comparison}
    \end{table}
    \item [3)] Rearranging \eqref{eq_Lip_uni} in \cref{Thm_Lipschitz_univariate}, we find for $\kappa=\sqrt{\frac{5}{3}}$
    \begin{align}
        \sum_{t\in Y^{\hat{\mu}_1}} |c_t-c_{\eta(t)}|^2+ N^2 \sum_{t\in Y^{\hat{\mu}_1}} \|t-\eta(t)\|_\T^2 \leq \max\left(\frac{5}{12\cdot c_{\min}^2},\frac{5}{2}\right) \sqrt{\frac{5}{3}}\frac{\|\hat{\mu}_1-\hat{\mu}_2\|_2^2}{N+1}  . \label{eq_Lipschitz_local}
    \end{align}
    This can be understood as a local Lipschitz property of the Prony-map $\mathscr{P}$ since the estimate holds for all $\hat{\mu}_1,\hat{\mu}_2\in \widehat{\mathcal{M}}_{c_{\min}}^N(\frac{2\kappa}{N})$ and a distance bounded in \eqref{eq_cond_univ}.\footnote{Motivated by the left hand side of estimates like \eqref{eq_Lipschitz_local}, we might view this as a weighted $\ell^2$-norm on the parameter set even if the expression depends on $|Y^{\hat{\mu}_1}|$. For the local results, we refuse to define a proper metric for the parameter set but study estimates like \eqref{eq_Lipschitz_local}.} Unsurprisingly, the problem becomes easier and hence the Lipschitz constant smaller if one has more samples $N$ and a larger lower bound for the weights.
    \item [4)] \Cref{Thm_Lipschitz_univariate} enables a comparison with \cref{Rem_Batenkov_2020}. Trivially, we have
    \begin{align*}
        \|\hat{\mu}_1(k)-\hat{\mu}_2(k)\|_2^2 \leq (2N+1) \|\hat{\mu}_1-\hat{\mu}_2\|_{\infty}^2 
    \end{align*}
    and 
    \begin{align*}
        \sum_{t\in Y^{\hat{\mu}_1}} \|t-\eta(t)\|_\T^2 \geq \md(\supp \mathscr{P}(\hat{\mu}_1),\supp \mathscr{P}(\hat{\mu}_2))^2.
    \end{align*}
    Applying this to \eqref{eq_Lipschitz_local}, we end up with
    \begin{align}
        \md(\supp \mathscr{P}(\hat{\mu}_1),\supp \mathscr{P}(\hat{\mu}_2)) &\leq \sqrt{\max\left(\frac{5}{6\cdot c_{\min}^2},5\right)}\left(\frac{5}{3}\right)^{1/4}\frac{\|\hat{\mu}_1-\hat{\mu}_2\|_{\infty}}{N}.   
        \label{eq_compare_to_Batenkov}
    \end{align}
    This reveals the same order in $N$ as in \eqref{eq_Batenkov_2020} and indicates some sense of optimality in the orders of Diederichs and our result. Moreover, the two presented theorems contain explicit bounds on the difference of the moments vectors guaranteeing their validity and enable to specify a reasonable estimate for the arising Lipschitz constant which is both not provided explicitly by \cite{Batenkov_2020}. 
    \item [5)] The proof of \cref{Prop_lower_bound} relies on the properties of a suitable localising function and on Poisson's summation formula. We construct this localising function in the next subsection.
\end{itemize}
\end{Rem}

\subsection{Local bivariate results}\label{subsec_bivariate}
The task in higher dimensions is to recover the node set $Y\subset \T^d,\,d\in\N,\,d\geq 1$ and weights $(c_t)_{t\in Y}\in \C^{|Y|}$ of an exponential sum
\begin{align*}
    \hat{\mu}(k)=\sum_{t\in Y} c_t \eim{t\cdot k}
\end{align*}
from its evaluations on a sampling set $\mathcal{B}\subset \Z^d$. Analogously to the univariate situation, we define the space of $q$-separated probability-like measures
\begin{align*}
    \mathcal{M}_{c_{\min}}^d(q) := \left\{\sum_{t\in Y} c_t \delta_{t},\,c_t\in\C,\,|c_t|\geq c_{\min},\,\sum_{t\in Y} c_t =1,\,Y\subset \T^d \text{ finite},\,\sep\, Y\geq q\right\} \subset \mathcal{M}(\T^d)
\end{align*}
and its corresponding moment space
\begin{align*}
    \widehat{\mathcal{M}}_{c_{\min}}^{N,d}(q) := \left\{\left(\sum_{t\in Y} c_t \eim{t k}\right)_{k\in \mathcal{B}},\, \sum_{t\in Y} c_t \delta_{t}\in \mathcal{M}_{c_{\min}}^d(q) \right\} \hspace{-0.05cm}\subset \C^{\mathcal{B}} 
\end{align*}
where 
\begin{align*}
    \sep\,Y:= \min_{t_1,t_2\in Y, t_1\neq t_2} \|t_1-t_2\|_{\T^d} = \min_{t_1,t_2\in Y, t_1\neq t_2} \min_{j\in\Z^d} \|t_1-t_2+j\|_{\infty}
\end{align*}
is the minimal separation of the set $Y$ and $\C^{\mathcal{B}}:=\left\{\left(z_k\right)_{k\in\mathcal{B}}: z_k\in\C\right\}$. As in the univariate case, one can apply techniques including the usage of a suitable localisation function and of Poisson's theorem. But naturally, it is crucial to develop ideas for the construction of suitable localising functions depending on the type of the set $\mathcal{B}$. Restricting the considered sampling sets $\mathcal{B}$ to $\ell^p$-balls in $\Z^d$ for $p\in[2,\infty)\cup \{\infty\}$, we have the following definition for the multivariate Prony-map.
\begin{Def}[Multivariate Prony-map]
Let $N\in \N$, $d\geq 2$ and $p\in[2,\infty)\cup \{\infty\}$. Assume we have moments $\hat{\mu}(k)$ for $k$ in the sampling set $\mathcal{B}=\{k\in\Z^d: \|k\|_p\leq N\}$. Then, the \textit{multivariate Prony-map} $\mathscr{P}$ is
\begin{align*}
    \mathscr{P}: \quad \widehat{\mathcal{M}}^{N,d}_{c_{\min}}\left(\frac{2\sqrt{d}}{N}\right) \to \mathcal{M}\left(\T^{d}\right), \quad \left(\sum_{t\in Y} c_t \eim{t k}\right)_{k\in \mathcal{B}} \mapsto \sum_{t\in Y} c_t \delta_t
\end{align*}
where $\mathcal{M}(\T^d)$ is the set of complex probability-like Borel measures on $\T^d$. 
\end{Def}
Again, the condition for the separation is motivated by our main results (\Cref{Thm_l2_lower_bound_2D} and \cref{Thm_l2_lower_bound_higherD}). To the best of our knowledge, a local Lipschitz property for the Prony-map with $d\geq 2$ has only been proven for $d=2$ and $p=\infty$: In \cite{Diederichs_18} the case $\mathcal{B}=[-N,N]^2$ is addressed.
\begin{Thm}  \textup{(cf.\,\cite[Thm.\,2.31.]{Diederichs_18})} \label{Thm_Diederichs_2D}Let $\hat{\mu}_1,\hat{\mu}_2\in \widehat{\mathcal{M}}^{N,2}_{c_{\min}}(\frac{2}{N+1})$ and assume
\begin{align*}
    \|\hat{\mu}_1-\hat{\mu}_2\|_2^2:=\sum_{\genfrac{}{}{0pt}{}{\|k\|_{\infty}\leq N}{k\in\Z^2}} |\hat{\mu}_1(k)-\hat{\mu}_2(k)|^2 < \frac{5}{4} (N+1)^2 c_{\min}^2.
\end{align*}
Then the following holds:
\begin{itemize}
    \item [(i)] For every node $t\in Y^{\hat{\mu}_1}$ there is exactly one $t'=\eta(t)\in Y^{\hat{\mu}_2}$ with $\|t-\eta(t)\|_{\T^2}<\frac{1}{2(N+1)}$ and vice versa.
    \item [(ii)] The exponential sums satisfy
    \begin{align*}
        \|\hat{\mu}_1-\hat{\mu}_2\|_2^2 \geq \frac{15}{16} (N+1)^4 \sum_{t\in Y^{\hat{\mu}_1}} \left(|c_t^{\hat{\mu}_1}|^2+|c_t^{\hat{\mu}_2}|^2\right) \|t-\eta(t)\|_{\T^2}^2 + \frac{3(N+1)^2}{4} \sum_{t\in Y^{\hat{\mu}_1}} |c_t^{\hat{\mu}_1}-c_t^{\hat{\mu}_2}|^2.
    \end{align*}
\end{itemize}
\end{Thm}
Hence, we conclude that the twodimensional Prony-map for rectangular samples satisfies a local Lipschitz property similar to the one-dimensional case. In the proof of the previous theorem, the tensor product of univariate minorising functions is added to a correction term in order to generate a bivariate localising function. But this is not only a difficult approach when we consider higher dimensions $d>2$ but also we might be interested in different sampling sets $\mathcal{B}$. For example, we might be given the samples on a $\ell^2$-ball in the case of an approximately radially symmetric problem like deconvolution for an optical system with a radial point spread function. As this was our underlying motivation, we try to obtain a result similar to \cref{Thm_Diederichs_2D} and using a construction for the localising function from \cite{Kunis_17}. Our final result is the following.
\begin{Thm} \label{Thm_l2_lower_bound_2D}
Let $\hat{\mu}_1,\hat{\mu}_2\in \widehat{\mathcal{M}}^{N,2}_{c_{\min}}(2\frac{\sqrt{2}}{N})$ and assume that $\hat{\mu}_1,\hat{\mu}_2$ satisfy
\begin{align} \label{eq_error_small}
    \sum_{\genfrac{}{}{0pt}{}{k\in\Z^2}{\|k\|_2\leq N}} |\hat{\mu}_1(k)-\hat{\mu}_2(k)|^2 < \frac{3}{4} N^2 c_{\min}^2.
\end{align}
Then for every $t\in Y^{\hat{\mu}_1}$ there is exactly one $t'=\eta(t)\in Y^{\hat{\mu}_2}$ with $\|t-\eta(t)\|_{\T^2}\leq \frac{1}{\sqrt{2}N}$ and vice versa. Moreover, we have the following estimate for the difference in the nodes and weights:
\begin{align*}
    \sum_{\genfrac{}{}{0pt}{}{k\in\Z^2}{\|k\|_2\leq N}} |\hat{\mu}_1(k)-\hat{\mu}_2(k)|^2 \geq \frac{5}{4} N^4 c_{\min}^2 \sum_{t\in Y^{\hat{\mu}_1}} \|t-\eta(t)\|_{\T^2}^2 + \frac{1}{16} N^2 \sum_{t\in Y^{\hat{\mu}_1}} |c_t^{(1)}-c_{\eta(t)}^{(2)}|^2.
\end{align*}
\end{Thm}
We remark that \cref{Thm_l2_lower_bound_2D} and \cref{Thm_Diederichs_2D} show the same order in $N$ and $\|t-\eta(t)\|_{\T^2}$ while the constants in \cref{Thm_Diederichs_2D} are favourable. Nevertheless, our own theorem is justified by its broader generality for the norm used for the sampling in Fourier space. Additionally, our construction allows a simple extension to arbitrary dimensions which we present in the next subsection. The rest of this subsection is intended to prove \cref{Thm_l2_lower_bound_2D} and we use the following construction.

\begin{Lem} \textup{(cf.\,\cite[Lemma 2.1]{Kunis_17})} \label{Prop_psi_generell}
Let $d, r\in\N$, $p=2r$  and $N,q>0$. Consider $\varphi:\R\to \R_{\geq 0}$ to be some continuous and even function with $\supp \varphi\subset[-q/2,q/2]$ and a $r$-th weak derivative of bounded variation. Then, we have for the function $\psi:\R^d \to \R$,
\begin{align*}
    \psi=\left((2\pi N)^p-(-1)^r \sum_{s=1}^d \frac{\partial^p}{\partial x_s^p}\right) \bigotimes_{\ell=1}^d \varphi *\varphi
\end{align*}
\begin{itemize}
    \item [(i)] $\supp \psi \subset [-q,q]^d$,
    \item [(ii)] its Fourier transform $\hat{\psi}(v):=\int_{\R^d} \psi(x) \eim{vx} dx$ obeys
    \begin{align*}
        \hat{\psi}(v)\begin{cases} \geq 0, &\quad \|v\|_p\leq N \\ \leq 0, &\quad \|v\|_p\geq N, \end{cases} 
    \end{align*}
    \item [(iii)] $\max_{v\in\R^d} \hat{\psi}(v)=\hat{\psi}(0)$
    and $\hat{\psi}\in L^1(\R^d)$.\footnote{Condition (i) and (iv) allow to apply the Poisson summation formula.}
\end{itemize}
\end{Lem}
\begin{proof}
\begin{itemize}
    \item [(i)] The first fact follows from the well-known property $\supp(\varphi*\varphi)\subset \supp\varphi + \supp\varphi$.
    \item [(ii)] Using the correspondence of differentiation and multiplication by polynomials through the Fourier transform, one has
    \begin{align*}
        \hat{\psi}(v)=\left((2\pi N)^p - \sum_{s=1}^d (2\pi v_s)^p \right) \prod_{\ell=1}^d \left(\hat{\varphi}(v_\ell)\right)^2
    \end{align*}
    where we also applied the Fourier-convolution theorem. This leads to the second assertion.
    \item [(iii)] Nonnegativity of $\varphi$ yields 
    \begin{align*}
        |\hat{\varphi}(v)|\leq \int_{\R^d} |\varphi(x)| |\eim{vx}| dx = \int_{\R^d} \varphi(x) dx = \hat{\varphi}(0) \quad \text{for all }v\in\R^d
    \end{align*}
    and by the parity of $\varphi$ we have $\hat{\varphi}(v)\in\R$. Hence, we have $\hat{\varphi}(v)\leq \hat{\varphi}(0)$. The representation of $\hat{\psi}$ from the proof of (ii) leads to the estimate
    \begin{align*}
        \hat{\psi}(v)\leq (2\pi N)^p \prod_{\ell=1}^d \left(\hat{\varphi}(v_\ell)\right)^2 \leq (2\pi N)^p \prod_{\ell=1}^d \left(\hat{\varphi}(0)\right)^2 = \hat{\psi}(0) \quad \text{for every $v\in\R^d$.}
    \end{align*}
     Finally, the smoothness of $\varphi$ yields $|\hat{\varphi}(v)|\leq C (1+|v|)^{-r-1}$ and by the representation of $\hat{\psi}$ we observe that $\hat{\psi}$ has sufficient decay.
\end{itemize}
\vspace{-0.4cm}
\end{proof}
\begin{figure}[ht]
  \centering
  \subfloat[][]{\includegraphics[width=0.47\linewidth]{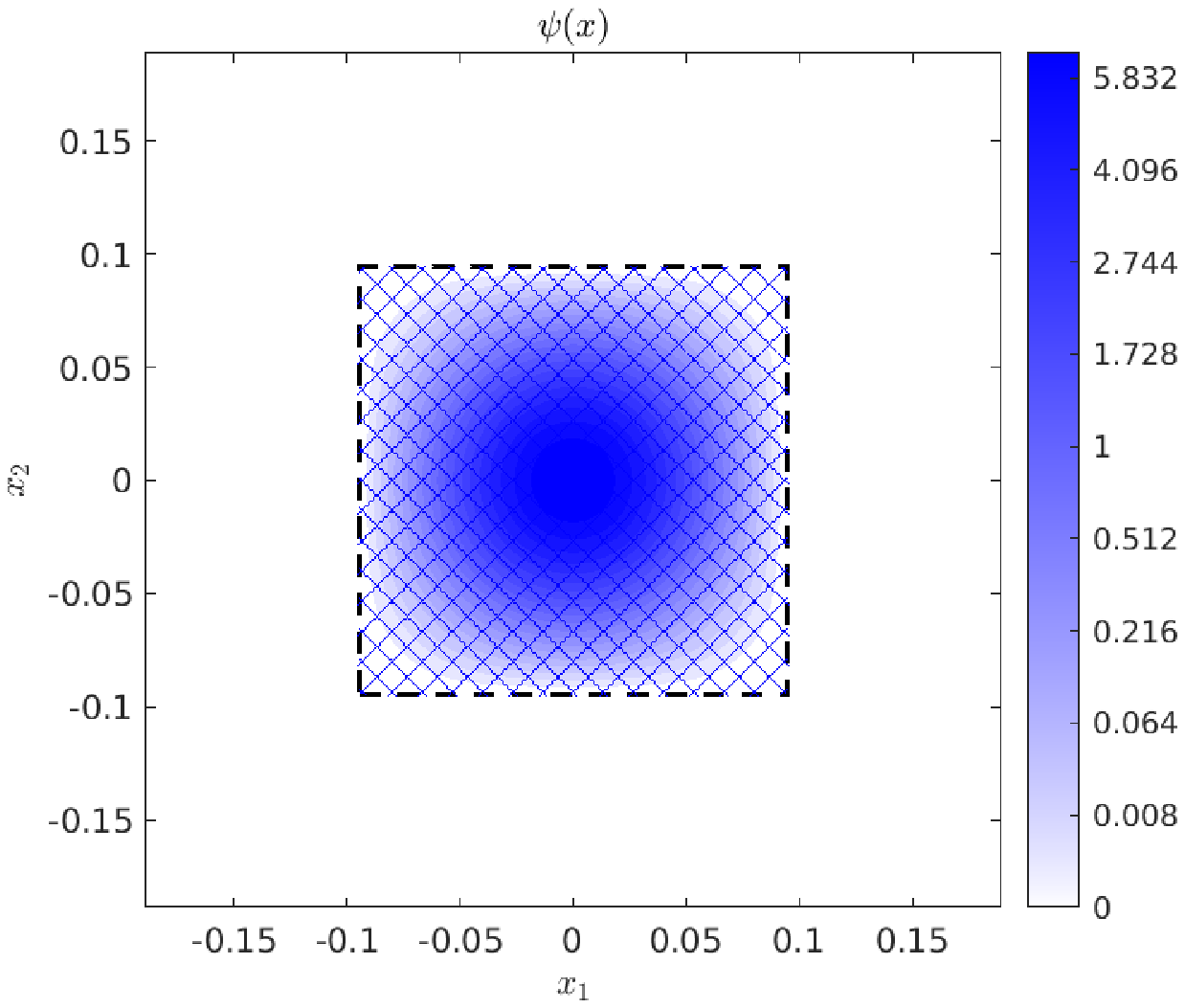}}
  \qquad
  \subfloat[][]{\includegraphics[width=0.47\linewidth]{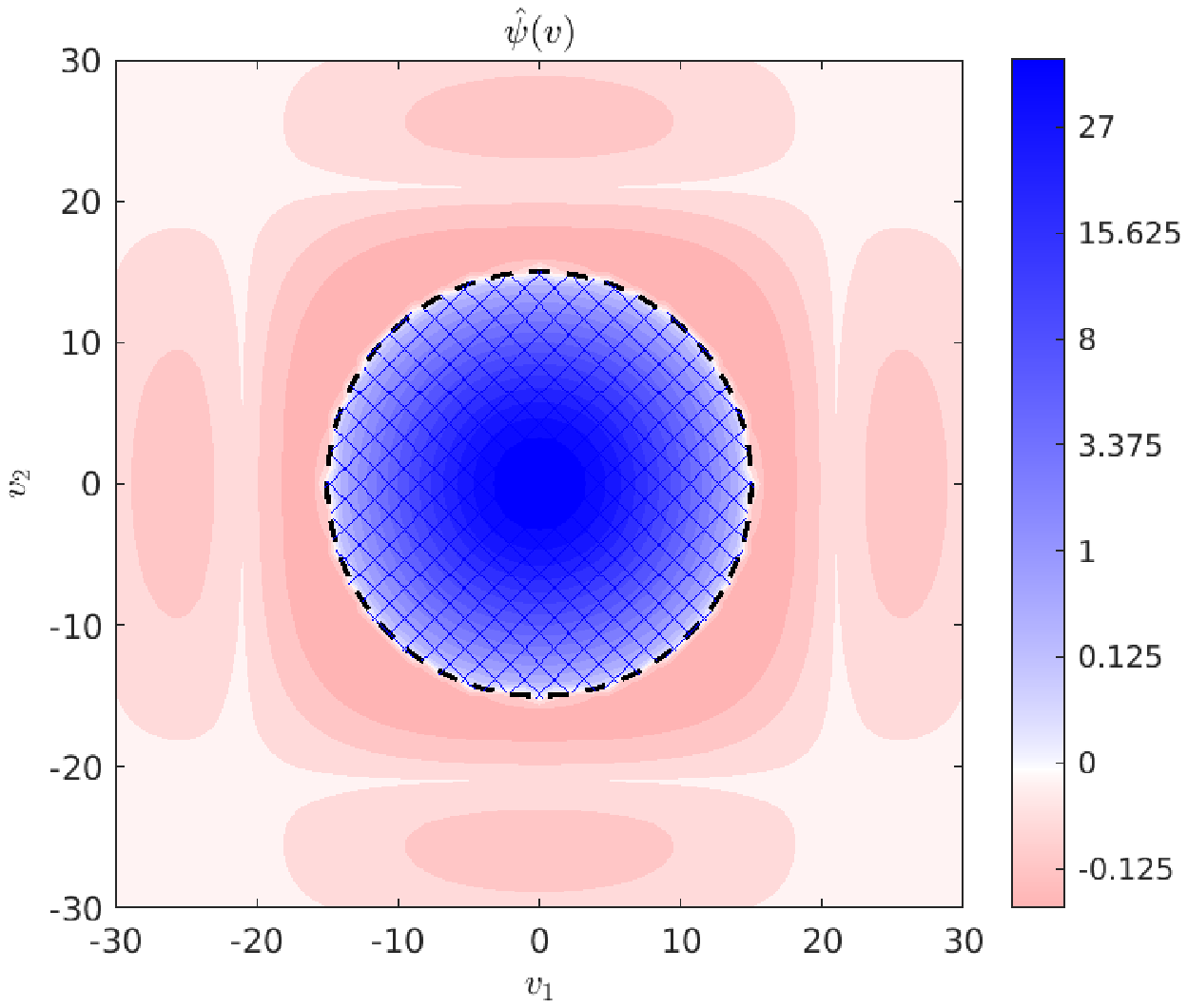}}
  \caption{Localising function $\psi$ and its Fourier transform $\hat{\psi}$ for $\varphi$ as chosen in \eqref{eq_cos_squared} and $N=15$, $q=\frac{\sqrt{2}}{N}$. The function $\psi$ is nonnegative inside of the dashed, hatched rectangle and zero outside (a) while $\hat{\psi}$ is nonnegative inside the dash-dotted, hatched circle and nonpositive outside (b). Both functions are maximal in zero for this choice of $\varphi$.}
    \label{fig:psi_psi_hat}
\end{figure}
Compactly supported, symmetric functions like $\varphi$ are often called \textit{window functions} in signal processing and statistics (cf.\,\cite{Plonka_18}). Based on this concept, we refer to $\varphi$ as a window function or window. From now on, we concentrate on the $\ell^2$-case where $p=2$. Note that choosing a different $p$ might require a different window function $\varphi$. Nevertheless, this approach does not limit the applicability of the result due to the fact that having samples on $\mathcal{B}=\{k\in\Z^d: \|k\|_p\leq N\}$ for $p>2$ implies knowledge about the moments on $\{k\in\Z^d: \|k\|_2\leq N\}$.\par  
This subsection deals only with the case of $d=2$ using a localising function $\psi$ depicted in \cref{fig:psi_psi_hat} with 
\begin{align} \label{eq_cos_squared}
    \varphi(x)=\begin{cases} \cos^2\left(\frac{\pi x}{q}\right), &\quad |x|<\frac{q}{2},\\ 0, &\quad \text{otherwise.} \end{cases}
\end{align}
as our choice for the window function. Note that this function was also considered (up to a factor of 2) in \cite[Rem.\,2.3]{Kunis_17} but for the case $p=4$. One can identify our choice of $\varphi$ as the \textit{Hann (or Hanning)-window} (cf.\,\cite[p.\,95]{Plonka_18}). In order to explain our choice of $\varphi$, we emphasize that for statements like \cref{Prop_lower_bound} and \cref{Thm_Diederichs_2D} we need $\psi$ to have its global maximum in 0. Even further, we will need a positive lower bound for the decay $\psi(0)-\psi(x)$ in terms of $\|x\|_{\T^2}$ later. We show in the two following remarks that two other typical choices of $\varphi$ fail to have this property.
\begin{Rem} \label{Rem_phi_1}
The function $\psi$ does not have a maximum in zero, if we choose $\varphi=\varphi_1$ where $\varphi_1$ is the function
\begin{align*}
    \varphi_1(x)=\begin{cases} 1-\left(\frac{2x}{q}\right)^2, &\quad |x|\leq \frac{q}{2} \\ 0,& \quad \text{otherwise} \end{cases}
\end{align*}
which is used in \cite{Kunis_17}. We can conclude from
\begin{align*}
    \varphi_1*\varphi_1(x)\leq \int_{\R} |\varphi_1(y)| |\varphi_1(x-y)| dx \leq \|\varphi_1\|_2^2 = \varphi_1*\varphi_1(0)
\end{align*}
that $\varphi_1*\varphi_1$ has a local and global maximum at 0. Hence, the function $\varphi_1*\varphi_1$ has a vanishing derivative at $x=0$ and this means
\begin{align*}
    \varphi_1*\varphi_1(\epsilon)=\varphi_1*\varphi_1(0)+(\varphi_1*\varphi_1)'(0)\epsilon + o(\epsilon) = \frac{8q}{15} + o(\epsilon)
\end{align*}
for $\epsilon>0$. Because $\varphi_1$ has the weak derivative $\varphi_1'(x)=\frac{-8x}{q^2}$, it is straightforward to compute 
\begin{align*} 
(\varphi_1*\varphi_1)''(\epsilon)=(\varphi_1'*\varphi_1')(\epsilon)=-\frac{16}{3q} +\frac{16}{q^2}\epsilon+o(\epsilon)
\end{align*}
for $\epsilon>0$. Plugging this into the definition of $\psi$, we obtain
\begin{align*}
    \psi(\epsilon,\epsilon)&=(2\pi N)^2 \left(\varphi_1*\varphi_1(\epsilon)\right)^2 + 2 (\varphi_1*\varphi_1)''(\epsilon) (\varphi_1*\varphi_1)(\epsilon) \\
    &=(2\pi N)^2(\varphi_1*\varphi_1)(0)^2 + 2 (\varphi_1*\varphi_1)''(0) (\varphi_1*\varphi_1)(0) +\frac{16^2}{15} \epsilon+o(\epsilon) \\
    &=\psi(0,0)+\frac{16^2}{15} \epsilon+o(\epsilon).
\end{align*}
Consequently, $\psi$ can not have a maximum in 0.
\end{Rem}
\begin{Rem} In \cite[Rem.\,2.2]{Kunis_17} the alternative
\begin{align*}
    \varphi_2(x)=\begin{cases} \cos\left(\frac{\pi x}{q}\right), &\quad |x|<\frac{q}{2},\\ 0, &\quad \text{otherwise} \end{cases}
\end{align*}
is mentioned and this function has $\varphi_2''=-\pi^2 q^{-2} \varphi_2 + \pi q^{-1} \left(\delta_{q/2}+\delta_{-q/2}\right)$ as its second distributional derivative. But again, the resulting $\psi$ does not have a global maximum in 0. Using the same technique as before, we have
\begin{align*}
    \varphi_2*\varphi_2(\epsilon)=\varphi_2*\varphi_2(0)+(\varphi_2*\varphi_2)'(0) \epsilon + o(\epsilon) = \frac{q}{2} + o(\epsilon).
\end{align*}
Applying the expression for the second distributional derivative, one finds
\begin{align*}
    (\varphi_2*\varphi_2)''(\epsilon)&=-\frac{\pi^2}{q^2} \varphi_2*\varphi_2(\epsilon) + \frac{\pi}{q}\left(\varphi_2(\epsilon-q/2)+\varphi_2(\epsilon+q/2)\right) \\
    &=-\frac{\pi^2}{2q}+\frac{\pi}{q}\sin\left(\frac{\pi|\epsilon|}{q}\right) + o(\epsilon) \\
    &=-\frac{\pi^2}{2q}+\frac{\pi^2}{q^2}|\epsilon|+o(\epsilon).
\end{align*}
Therefore, we end up with $\psi(\epsilon,\epsilon)=\psi(0,0)+\frac{\pi^2}{q} |\epsilon| +o(\epsilon)$ and $\psi$ can not have a maximum at the origin.
\end{Rem}
In both attempts and especially for $\varphi_2$, the non-smoothness at the boundary of the support was the problem causing the appearance of additional terms which do not allow the function to be maximal in 0. This motivated us to consider $\varphi$ as in \eqref{eq_cos_squared}. We display the attempted window functions in \cref{fig:phi_and_phi_convolved} where we also include $\varphi*\varphi$ explicitly computed in the following lemma.
\begin{figure}[ht]
    \centering
    \includegraphics[width=\textwidth]{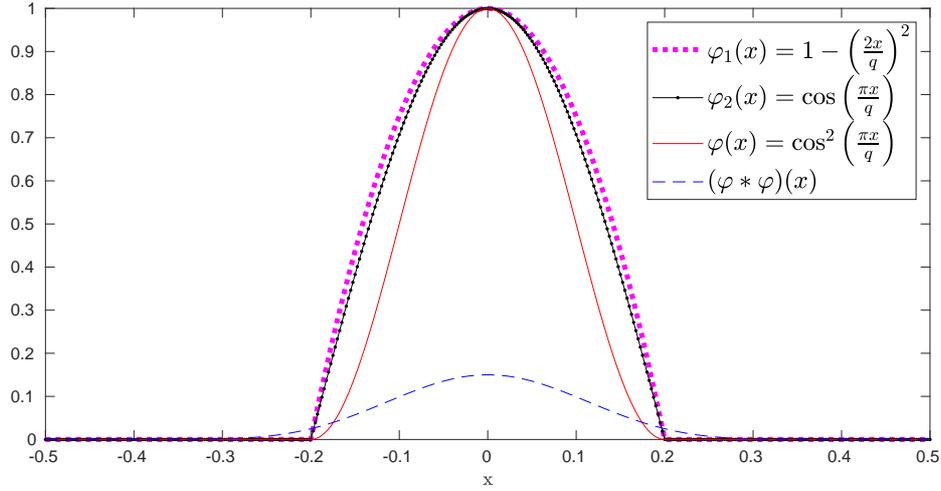}
    \caption{Window functions $\varphi_1,\varphi_2, \varphi$ and convolution $\varphi*\varphi$ for $q=0.4$. Note that $\varphi$ has higher regularity compared to $\varphi_1$ and $\varphi_2$ at $x=\pm \frac{q}{2}=\pm 0.2$.}
    \label{fig:phi_and_phi_convolved}
\end{figure}
\begin{Lem} \label{Lem_phi_explicit}
Let $\varphi$ be defined as in \eqref{eq_cos_squared}. Then we have
\begin{align*}
    \varphi*\varphi(x)=\begin{cases} \frac{q-|x|}{4} \left(1+\frac{1}{2}\cos\left(\frac{2\pi x}{q}\right)\right) + \frac{3}{8} \frac{q}{2\pi} \sin\left(\frac{2\pi |x|}{q}\right), &\quad |x|<q \\ 0,&\quad \text{otherwise}\end{cases}
\end{align*}
and
\begin{align*}
    (\varphi*\varphi)''(x)=-\frac{4\pi^2}{q^2} \varphi*\varphi(x) + \frac{\pi^2}{q^2} \left(\frac{q}{2\pi} \sin\left(\frac{2\pi|x|}{q}\right)+q-|x|\right)
\end{align*}
for $|x|<q$. Furthermore, the localising function $\psi$ taking the form
\begin{align*}
    \psi(x)&=\left[(2\pi N)^2-\frac{8\pi^2}{q^2}\right] (\varphi*\varphi)(x_1) \cdot (\varphi*\varphi)(x_2) \\
    &\,+\frac{\pi^2}{q^2} \left[\left(\frac{q}{2\pi} \sin\left(\frac{2\pi|x_1|}{q}\right)+q-|x_1|\right) \varphi*\varphi(x_2) + \left(\frac{q}{2\pi} \sin\left(\frac{2\pi|x_2|}{q}\right)+q-|x_2|\right) \varphi*\varphi(x_1) \right]
\end{align*}
has its global maximum in $(x_1,x_2)^\top=(0,0)^\top$ if $Nq\geq \sqrt{2}$.
\end{Lem}
\begin{proof}
It is clear that $\supp(\varphi*\varphi)\subset \supp \varphi + \supp \varphi = [-q,q]$. For $|x|<q$ one can observe that the common support of $\varphi$ and $\varphi(x-\cdot)$ is the interval $[\max(-\frac{q}{2},x-\frac{q}{2}),\min(\frac{q}{2},x+\frac{q}{2})]$. Therefore, well-known trigonometric identities yield
\begin{align*}
    \varphi*\varphi(x)&= \int_{\max(-q/2,x-q/2)}^{\min(q/2,x+q/2)} \cos^2\left(\frac{\pi y}{q}\right) \cos^2\left(\frac{\pi (x-y)}{q}\right) dy \\
    &=\frac{1}{4}\int_{\max(-q/2,x-q/2)}^{\min(q/2,x+q/2)} \left(1+\cos\left(\frac{2\pi y}{q}\right)\right) \left(1+\cos\left(\frac{2\pi (x-y)}{q}\right)\right) dy \\
    &=\frac{1}{4}\int_{\max(-q/2,x-q/2)}^{\min(q/2,x+q/2)} 1 + \cos\frac{2\pi y}{q} + \cos\frac{2\pi (x-y)}{q} + \frac{1}{2} \cos\frac{2\pi x}{q} + \frac{1}{2} \cos\frac{2\pi(2y-x)}{q} dy \\
    &=\frac{q-|x|}{4} \left(1+\frac{1}{2}\cos\left(\frac{2\pi x}{q}\right)\right) + \frac{3}{8} \frac{q}{2\pi} \sin\left(\frac{2\pi |x|}{q}\right)
\end{align*}
since all terms in the integrand can be integrated easily. Additionally, the function $\varphi$ has the integrable weak second derivative
\begin{align*}
    \varphi''(x)=-\frac{4\pi^2}{q^2} \varphi(x) + \frac{\pi^2}{q^2} \chi_{[-q/2,q/2]}(x)
\end{align*}
leading to
\begin{align}
    (\varphi*\varphi)''(x)&=-\frac{4\pi^2}{q^2} \varphi*\varphi(x) + \frac{\pi^2}{q^2} \int_{\max(-q/2,x-q/2)}^{\min(q/2,x+q/2)} \varphi(x-y) dy \nonumber\\
    &=-\frac{4\pi^2}{q^2} \varphi*\varphi(x) + \frac{\pi^2}{q^2} \left(q-|x|+ \left[\frac{q}{2\pi}\sin\frac{2\pi (y-x)}{q}\right]_{\max(-q/2,x-q/2)}^{\min(q/2,x+q/2)}\right) \nonumber\\
    &=-\frac{4\pi^2}{q^2} \varphi*\varphi(x) + \frac{\pi^2}{q^2} \left(\frac{q}{2\pi} \sin\left(\frac{2\pi|x|}{q}\right)+q-|x|\right). \label{eq_varphi_scd_deriv}
\end{align}
So we find
\begin{align*}
    \psi(x)&=(2\pi N)^2 (\varphi*\varphi)(x_1) \cdot (\varphi*\varphi)(x_2) + (\varphi*\varphi)''(x_1) \cdot \varphi*\varphi(x_2) + \varphi*\varphi(x_1) \cdot (\varphi*\varphi)''(x_2) \\
    &=\left[(2\pi N)^2-\frac{8\pi^2}{q^2}\right] (\varphi*\varphi)(x_1) \cdot (\varphi*\varphi)(x_2) \\
    &\,+\frac{\pi^2}{q^2} \left[\left(\frac{q}{2\pi} \sin\left(\frac{2\pi|x_1|}{q}\right) + q-|x_1|\right) \varphi*\varphi(x_2) + \left(\frac{q}{2\pi} \sin\left(\frac{2\pi|x_2|}{q}\right)+q-|x_2|\right) \varphi*\varphi(x_1) \right].
\end{align*}
The key observation is that 
\begin{align}
    \frac{q}{2\pi} \sin\left(\frac{2\pi|x|}{q}\right)+q-|x|\leq \frac{q}{2\pi} \frac{2\pi|x|}{q}+q-|x| = q = \frac{q}{2\pi} \sin\left(\frac{2\pi|0|}{q}\right)+q-|0| \label{eq_key_observe}
\end{align}
allows us to estimate\footnote{Note that $\varphi*\varphi$ is nonnegative and attains its maximum in 0 which can be proven analogously to the proof for $\varphi_1$ in \cref{Rem_phi_1}. Moreover, $\psi\geq 0$ since $\frac{1}{2\pi}\sin\left(\frac{2\pi |x|}{q}\right)+1-\frac{|x|}{q}\geq 0$ for $|x|\leq q$ by direct computation.}
\begin{align*}
    0\leq \psi(x)&\leq \left[(2\pi N)^2-\frac{8\pi^2}{q^2}\right] (\varphi*\varphi)(0) \cdot (\varphi*\varphi)(0) \\
    &\quad+\frac{\pi^2}{q^2} \left[\left(\frac{q}{2\pi} \sin\left(\frac{2\pi|0|}{q}\right)+q-|0|\right) \varphi*\varphi(0) + \left(\frac{q}{2\pi} \sin\left(\frac{2\pi|0|}{q}\right)+q-|0|\right) \varphi*\varphi(0) \right] \\
    &=\psi(0,0)
\end{align*}
if $(2\pi N)^2-\frac{8\pi^2}{q^2}\geq 0$. The latter is equivalent to $Nq\geq \sqrt{2}$.
\end{proof}
By the previous lemma, the most general choice for $q$ is $q=\frac{\sqrt{2}}{N}$ and for this we have
\begin{align*}
    \psi(0,0)&= \frac{\pi^2}{q^2} 2q(\varphi*\varphi)(0) = \frac{3\pi^2}{4}
\end{align*}
as the value of $\psi$ in the global maximum at zero. In order to find lower bounds for $\psi(0,0)-\psi(x)$ we need the following upper bounds for $\varphi*\varphi$.
\begin{Lem} \label{Lem_bound_conv}
Let $a\in \left(\frac{1}{2},1\right]$ and set up the difference quotient 
\begin{align}
    m=\frac{(\varphi*\varphi)(aq)-(\varphi*\varphi)(q/2)}{aq-q/2} \label{eq_Diff_Quotient}
\end{align} 
where $\varphi$ is defined as in \eqref{eq_cos_squared}. Then, one has
\begin{align*}
    (\varphi*\varphi)(x)\leq 
    \begin{cases} 
    \frac{3q}{8} - \frac{5}{4}\frac{|x|^2}{q} &,\quad |x|<\frac{q}{2}, \\
    \frac{q}{16}+\left(|x|-\frac{q}{2}\right)m &, \quad \frac{q}{2}\leq |x| \leq aq. 
    \end{cases}
\end{align*}
\end{Lem}
\begin{proof}
For the first inequality, we can consider by symmetry of $\varphi*\varphi$ that  $0\leq x\leq \frac{q}{2}$ and define the auxiliary function
\begin{align*}
    g(x)=\frac{3q}{8} - \frac{5}{4}\frac{|x|^2}{q}-(\varphi*\varphi)(x)
\end{align*}
for $x\in\left[0,\frac{q}{2}\right]$. By construction, one has $g(0)=g(q/2)=0$, $g'(x)=0$ and $g''(0)>0$ yielding a local minimum of $g$ in 0. Using \eqref{eq_varphi_scd_deriv} we can also compute the third derivative
\begin{align*}
    g^{(3)}(x)=-(\varphi*\varphi)^{(3)}(x)=-\frac{\pi^3(q-x)}{q^3} \sin\left(\frac{2\pi x}{q}\right)\leq 0
\end{align*}
for all $x\in \left[0,\frac{q}{2}\right]$. Hence, $g''$ is monotonically decreasing and there can be at most one local extreme point in $g'$. By Rolle's theorem, there is always a local extremum between two zeros of a differentiable function like $g'$. Therefore, $g'$ can have at most two zeros and $g$ itself cannot have an additional local minimum in $\left(0,\frac{q}{2}\right)$. So we have $g(x)\geq g(0)=g(q/2)=0$.\par
For the second part, we reuse our computation of the third derivative of $\varphi*\varphi$ and notice that it is negative on $\left[\frac{q}{2},q\right]$. Consequently, we have $(\varphi*\varphi)''(x)\geq (\varphi*\varphi)''(q)=0$ meaning convexity of $\varphi*\varphi$ on $\left[\frac{q}{2},q\right]$. This leads to the statement by noting that the upper bound is the linear interpolation of $\varphi*\varphi$ between $\frac{q}{2}$ and $aq\in\left(\frac{q}{2},q\right]$.
\end{proof}
\begin{Lem} \label{Lem_lower_bound_for_difference}
Let $Nq= \sqrt{2}$. Then, the localising function $\psi$ admits
\begin{align*}
    \psi(0)-\psi(x)\geq 
    \begin{cases} 
    \frac{5}{4}\frac{\pi^2}{q^2} \|x\|_{\infty}^2 &,\quad \|x\|_{\infty}\leq\frac{q}{2},\\
    \frac{1-m}{2} \frac{\pi^2}{q} \|x\|_{\infty} &,\quad \frac{q}{2}\leq \|x\|_{\infty}\leq aq
    \end{cases}
\end{align*}
for $a,m$ as in \cref{Lem_bound_conv}.
\end{Lem}
\begin{proof}
The \cref{Lem_bound_conv,Lem_phi_explicit} together with \eqref{eq_key_observe} and the symmetry of $\varphi$ imply in the case of $x$ with $\|x\|_{\infty}\leq\frac{q}{2}$
\begin{align*}
    \psi(0)-\psi(x)&\geq\frac{3\pi^2}{4} - \frac{\pi^2}{q^2} \left[q(\varphi*\varphi)(0)+q(\varphi*\varphi)(\|x\|_{\infty})\right] \nonumber\\
    &\geq \frac{3\pi^2}{8} - \frac{\pi^2}{q} \left(\frac{3q}{8} - \frac{5}{4}\frac{\|x\|_{\infty}^2}{q}\right) \\
    &=\frac{5}{4}\frac{\pi^2}{q^2} \|x\|_{\infty}^2.
\end{align*}
Note that $\frac{q}{2\pi}\sin\left(\frac{2\pi x}{q}\right)+q-|x|\leq \frac{q}{2}$ for $x\in\left[\frac{q}{2},q\right]$. Hence, we obtain for $\|x\|_{\infty}\in\left[\frac{q}{2},q\right]$
\begin{align*}
    \psi(0)-\psi(x)&\geq \frac{3\pi^2}{4} - \frac{\pi^2}{q^2} \left[\frac{q}{2}(\varphi*\varphi)(0)+q(\varphi*\varphi)(\|x\|_{\infty})\right] \nonumber\\
    &\geq \frac{3\pi^2}{4} - \frac{\pi^2}{q} \left[\frac{3q}{16} +  \frac{q}{16}+\left(\|x\|_{\infty}-\frac{q}{2}\right)m \right] \\
    &= \frac{\pi^2}{2} + \frac{\pi^2 m}{2} - \frac{\pi^2 m\|x\|_{\infty}}{q} \\
    &\geq \left(\frac{1}{2}+\frac{m}{2}-m\right) \frac{\|x\|_{\infty}}{q} = \frac{1-m}{2} \frac{\pi^2}{q} \|x\|_{\infty}
\end{align*}
using $\|x\|_{\infty}\leq aq\leq q$ for the last inequality.
\end{proof}
\begin{figure}[ht]
    \centering
    \includegraphics[width=\textwidth]{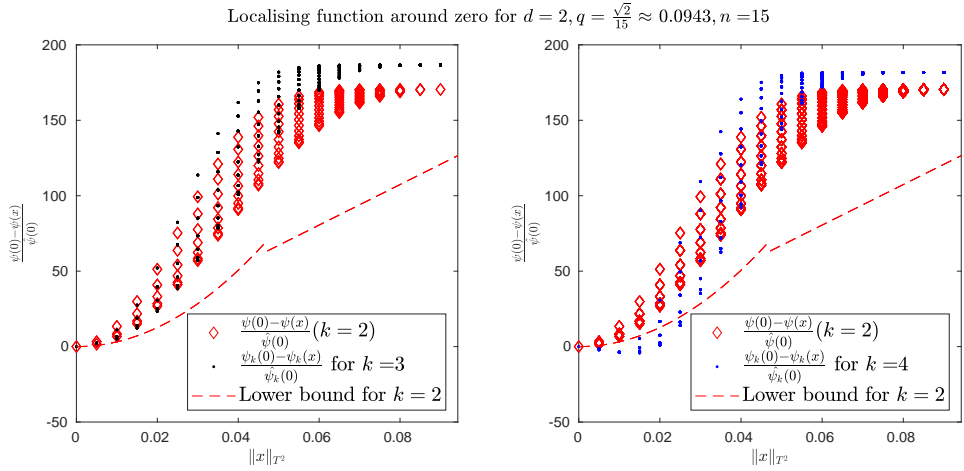}
    \caption{Behaviour of the localising function around zero (red) and the lower bound from \cref{Lem_lower_bound_for_difference} (red dashed) compared to $\frac{\psi_k(0)-\psi_k(x)}{\hat{\psi_k}(0)}$ for $k=3$ (black) and $k=4$ (blue).}
    \label{fig:localising_around_zero}
\end{figure}
We display $\frac{\psi(0)-\psi(x)}{\hat{\psi}(0)}$ and its lower bound from the previous \cref{Lem_lower_bound_for_difference} for $a=1$, $m=\frac{-1}{8}$ in \cref{fig:localising_around_zero}.\footnote{The normalisation by $\hat{\psi}(0)$ is motivated by the proof of the next theorem. The functions $\psi$ as well as the numerically computed $\psi_k$ were sampled on a grid and then $\|x\|_{\T^2}$ was computed for every grid point in order to generate the plot.} One might come up with the idea to consider higher powers $\cos^k\left(\frac{\pi x}{q}\right), k=3,4,\dots$, for the window function $\varphi$ in order to obtain higher regularity. The resulting localising function are called $\psi_k$ and they are included in \cref{fig:localising_around_zero} for $k=3,4$ where we computed these functions numerically. On one hand, one observes that the lower bound from \cref{Lem_lower_bound_for_difference} cannot be improved very much. On the other hand, choosing $k=3$ cannot be so much better than our choice $k=2$ and for $k\geq 4$ we do not even have a maximum of $\psi_k(x)$ for $x=0$. So we use $\varphi$ and hence $\psi$ as discussed in the previous lemmas to prove \cref{Thm_l2_lower_bound_2D}. In the course of this, we use the following Lemma proven in the appendix:
\begin{Lem}\label{Lem_sequence_of_bounds}
Under the assumptions of \cref{Thm_l2_lower_bound_2D}, we have $\|t-\eta(t)\|_{\T^2}\leq\frac{\sqrt{2}}{2N}$ for all $t\in Y^{\hat{\mu}_1}$. 
\end{Lem}
\begin{proof}[Proof of \cref{Thm_l2_lower_bound_2D}]
\begin{figure}[ht]
    \centering
    \begin{tikzpicture}[>=triangle 45,x=6cm,y=6cm,scale=1]
    \draw[->] (0,0) --node[below]{$x_1$} (1.1,0);
    \draw[->] (0,0) node[below]{0} -- node[left]{$x_2$} (0,1.1);
    \draw[-] (1,-2pt) node[below]{1} -- (1,2pt);
    \draw[-] (-2pt,1) node[left]{1} -- (2pt,1);
    \foreach \x in {(0.15,0.25),(0.5,0.45),(0.8,0.2)}{
    \draw[mark=*,mark size=2pt,mark options={color=blue}] plot coordinates {\x};
    \draw[color=black, very thick] {\x} circle (4pt);
    }
    \foreach \x in {(0.1,0.8)}{
    \draw[mark=*,mark size=2pt,mark options={color=blue}] plot coordinates {\x};
    \draw[color=red, dashed, very thick] {\x} circle (4pt);
    }
    \foreach \x in {(0.18,0.22),(0.55,0.47),(0.78,0.17)}{
    \draw[mark=triangle*,mark size=2pt,mark options={color=green}] plot coordinates {\x};
    \draw[color=black!100, dotted, very thick] {\x} circle (4pt);
    }
    \foreach \x in {(0.33,0.77),(0.9,0.6)}{
    \draw[mark=triangle*,mark size=2pt,mark options={color=green}] plot coordinates {\x};
    \draw[color=red, dashed, very thick] {\x} circle (4pt);
    }
    \draw[|-|] (0.9,1) -- node[below]{$\frac{\sqrt{2}}{N}$} (1,1);
    \node at (0.13,0.25)[above]{\small $t$};
    \node at (0.23,0.21)[below]{\small $\eta(t)$};
    \draw[mark=*,mark size=2pt,mark options={color=blue}] plot coordinates {(1.4,0.8)} node[right]{\,$Y^{\hat{\mu}_1}$};
    \draw[mark=triangle*,mark size=2pt,mark options={color=green}] plot coordinates {(1.4,0.7)} node[right]{\,$Y^{\hat{\mu}_2}$};
    \draw[color=black!100,dotted, very thick] (1.4,0.5) circle (4pt) node[right]{\,$Y_2$}; 
    \draw[color=black!100, very thick] (1.4,0.6) circle (4pt) node[right]{\,$Y_1$};
    \draw[color=red, dashed, very thick] (1.4,0.4) circle (4pt) node[black, right]{\,$Y_3$};
    \end{tikzpicture}
    \caption{Visualisation of the sets $Y^{\hat{\mu}_1}$ (blue circles) and $Y^{\hat{\mu}_2}$ (green triangles) as well as their subsets $Y_1\subset Y^{\hat{\mu}_1}$ (densely black circles) and $Y_2\subset Y^{\hat{\mu}_2}$ (dotted black circles). Nodes without a neighbour closer than $\frac{\sqrt{2}}{N}$ belong to $Y_3$ (dashed red circles).}
    \label{fig:vis_of_sets_Yj}
\end{figure}
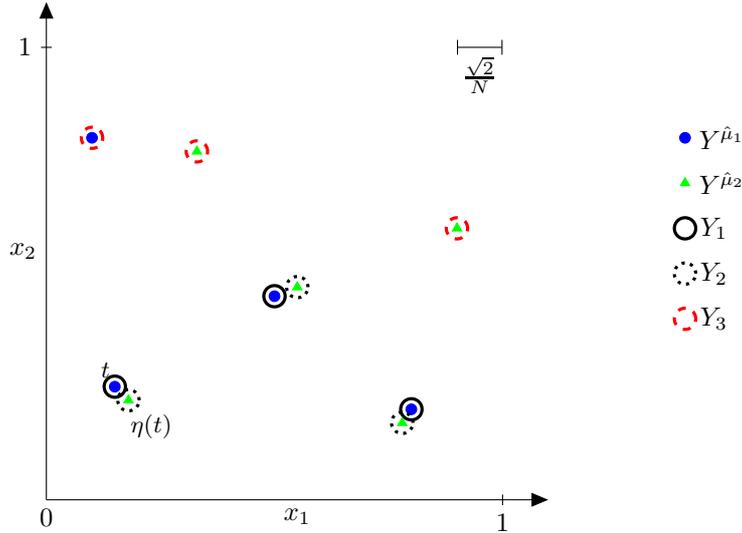
Let $q=\frac{\sqrt{2}}{N}$. Due to the condition on the separation of $\hat{\mu}_1$ and $\hat{\mu}_2$, we know that for every $t\in Y^{\hat{\mu}_1}$ there is at most one $t'\in Y^{\hat{\mu}_2}$ with $\|t-t'\|_{\T^2}<q$.\footnote{If there would be $t\in Y^{\hat{\mu}_1}$, $t_1,t_2\in Y^{\hat{\mu}_2}, t_1\neq t_2$ with $\|t-t_1\|_{\T^2}<q$ and $\|t-t_2\|_{\T^2}<q$, we have $\|t_1-t_2\|_{\T^2}<2q$ which is a contradiction to the assumption $\sep\,Y^{\hat{\mu}_2}\geq 2q$.}  We decompose the joint node set $Y:=Y^{\hat{\mu}_1}\cup Y^{\hat{\mu}_2}$ into $Y_1\subset Y^{\hat{\mu}_1}$, $Y_2\subset Y^{\hat{\mu}_2}$ and $Y_3\subset Y^{\hat{\mu}_1}\cup Y^{\hat{\mu}_2}$ with:\footnote{We adapted this construction from \cite[Thm.\,3.6]{Diederichs_19}.}

\begin{itemize}
    \item [(i)] For all $t,t'\in Y_j$ with $t\neq t'$ we have $\|t-t'\|_{\T^2}\geq \frac{\sqrt{2}}{N}$.
    \item [(ii)] For all $t\in Y_1$ there is exactly one $\eta(t)\in Y_2$ with $\|t-\eta(t)\|_{\T^2}<\frac{\sqrt{2}}{N}$.
    \item [(iii)] $Y_3:=\left\{t\in Y: \text{For all } t'\in Y \text{ with } t\neq t' \text{ one has } \|t-t'\|_{\T^2}\geq \frac{\sqrt{2}}{N}\right\}$ 
\end{itemize}

As one can compute that $\hat{\varphi}\in O\left(v^{-3}\right)$, the functions $\psi$ (having compact support) and $\hat{\psi}$ decrease fast enough in order to apply the Poisson summation formula. This together with \cref{Prop_psi_generell} and the convention
\begin{align*}
    \tilde{c}_t=
    \begin{cases}
    c_t^{(1)}, &\quad t\in Y^{\hat{\mu}_1}, \\
    -c_t^{(2)}, &\quad t\in Y^{\hat{\mu}_2},
    \end{cases}
\end{align*}
gives
\begin{align*}
    \hat{\psi}(0) \sum_{\genfrac{}{}{0pt}{}{k\in\Z^2}{\|k\|_2\leq N}} |\hat{\mu}_1(k)-\hat{\mu}_2(k)|^2 &= \max_{v\in\R^2} \hat{\psi}(v) \sum_{\genfrac{}{}{0pt}{}{k\in\Z^2}{\|k\|_2\leq N}} |\hat{\mu}_1(k)-\hat{\mu}_2(k)|^2 \\
    &\geq \sum_{k\in\Z^2} |\hat{\mu}_1(k)-\hat{\mu}_2(k)|^2 \hat{\psi}(k) \\
    &= \sum_{t,t'\in Y} \tilde{c}_t \overline{\tilde{c}_{t'}} \sum_{k\in\Z^2} \eim{k(t-t')} \hat{\psi}(k) \\
    &= \sum_{t,t'\in Y} \tilde{c}_t \overline{\tilde{c}_{t'}} \sum_{\ell\in\Z^2} \psi(t-t'+\ell) \\
    &= \psi(0) \sum_{t\in Y_3} |\tilde{c}_t|^2 + \sum_{t\in Y_1} \mathbf{c}_t^* \mathbf{A}_t \mathbf{c}_t
\end{align*}
where
\begin{align*}
    \mathbf{c}_t=\begin{pmatrix} \tilde{c}_t \\ \tilde{c}_{\eta(t)} \end{pmatrix} = \begin{pmatrix} c_t^{(1)} \\ -c_{\eta(t)}^{(2)} \end{pmatrix} \quad \text{,} \quad \mathbf{A}_t= \begin{pmatrix} \psi(0) & \psi(|t-\eta(t)|_{\T^2}) \\ \psi(|t-\eta(t)|_{\T^2}) & \psi(0) \end{pmatrix}
\end{align*}
and
\begin{align*}
    |t-t'|_{\T^d}=\left(\min_{\ell_1\in\Z} |t_1-t_1'+\ell_1|, \min_{\ell_2\in\Z} |t_2-t_2'+\ell_2|,\dots, \min_{\ell_d\in\Z} |t_d-t_d'+\ell_d|\right)^\top\in \left[0,\frac{1}{2}\right]^d.
\end{align*}
Using assumption \eqref{eq_error_small}, we obtain
\begin{align*}
    \frac{3}{4} N^2 c_{\min}^2 &> \sum_{\genfrac{}{}{0pt}{}{k\in\Z^2}{\|k\|_2\leq N}} |\hat{\mu}_1(k)-\hat{\mu}_2(k)|^2 \geq \frac{\psi(0)}{\hat{\psi}(0)} c_{\min}^2  |Y_3|\\
    &= \frac{\frac{3}{4}\pi^2}{4\pi^2 N^2 \hat{\varphi}(0)^4} c_{\min}^2 |Y_3| = \frac{3}{16 N^2 (q/2)^4} c_{\min}^2 |Y_3| = \frac{3}{2 q^2} c_{\min}^2 |Y_3|
\end{align*}
and this is equivalent to
\begin{align*}
    |Y_3| < \frac{2 q^2}{4} N^2 = 1
\end{align*}
meaning $Y_3=\emptyset$. So we know already that for all $t\in Y_1=Y^{\hat{\mu}_1}$ there is $t'=\eta(t)\in Y_2=Y^{\hat{\mu}_2}$ with $\|t-\eta(t)\|_{\T^2}<q=\frac{\sqrt{2}}{N}$. The eigendecomposition of the matrices $\mathbf{A}_t$ leads then to
\begin{align}
     \hspace{-0.2cm}\hat{\psi}(0) \hspace{-0.3cm}\sum_{\genfrac{}{}{0pt}{}{k\in\Z^2}{\|k\|_2\leq N}} |\hat{\mu}_1(k)-\hat{\mu}_2(k)|^2 &\geq \sum_{t\in Y_1} \left(\psi(0)-\psi(|t-\eta(t)|_{\T^2})\right) \mathbf{c}_t^* \frac{1}{\sqrt{2}}\begin{pmatrix} 
     1 \\ -1
     \end{pmatrix}
     \frac{1}{\sqrt{2}}
     \begin{pmatrix} 
     1 & -1
     \end{pmatrix}
     \mathbf{c}_t \nonumber\\
     & \quad \quad+ \left(\psi(0)+\psi(|t-\eta(t)|_{\T^2})\right) \mathbf{c}_t^* \frac{1}{\sqrt{2}}\begin{pmatrix} 
     1 \\ 1
     \end{pmatrix}
     \frac{1}{\sqrt{2}}
     \begin{pmatrix} 
     1 & 1
     \end{pmatrix}
     \mathbf{c}_t \nonumber\\
     &= \sum_{t\in Y_1}  \frac{1}{2} \left\{\psi(0)-\psi(|t-\eta(t))|_{\T^2}\right\} \left|c^{(1)}_t+c_{\eta(t)}^{(2)}\right|^2 \nonumber\\
     &\quad + \frac{1}{2} \left\{\psi(0)+\psi(|t-\eta(t)|_{\T^2})\right\} \left|c^{(1)}_t-c_{\eta(t)}^{(2)}\right|^2 \nonumber\\
     &= \sum_{t\in Y_1} \left(\psi(0)-\psi(|t-\eta(t)|_{\T^2}\right) \|\mathbf{c}_t\|_2^2 + \psi(|t-\eta(t)|_{\T^2}) \left|c^{(1)}_t-c_{\eta(t)}^{(2)}\right|^2. \label{eq_eigendecomposition}
\end{align}
In order to bound $\psi(|t-\eta(t)|_{\T^2})$ in the second term from below, we use \cref{Lem_sequence_of_bounds} and compute
\begin{align*}
    \psi(|t-\eta(t)|_{\T^2})\geq \psi\left(\frac{q}{2},\frac{q}{2}\right)=\frac{\pi^2}{16}.
\end{align*}
By application of \cref{Lem_lower_bound_for_difference} we finally conclude
\begin{align*}
    \sum_{\genfrac{}{}{0pt}{}{k\in\Z^2}{\|k\|_2\leq N}} |\hat{\mu}_1(k)-\hat{\mu}_2(k)|^2 &\geq \frac{1}{\hat{\psi}(0)} \sum_{t\in Y^{\hat{\mu}_1}} \frac{5\pi^2}{4q^2} \|t-\eta(t)\|_{\T^2}^2 \|\mathbf{c}_t\|_2^2 + \frac{\pi^2}{16} \left|c^{(1)}_t-c_{\eta(t)}^{(2)}\right|^2\\
    &=\frac{5}{4} N^4 c_{\min}^2 \sum_{t\in Y^{\hat{\mu}_1}} \|t-\eta(t)\|_{\T^2}^2 + \frac{1}{16} N^2 \sum_{t\in Y^{\hat{\mu}_1}} |c_t^{(1)}-c_{\eta(t)}^{(2)}|^2
\end{align*}
and this was the second claim.
\end{proof}

\subsection{Local results in higher dimension} \label{subsec_multi}
As the main result of this subsection, we extend our ideas from the previous subsection and obtain a local Lipschitz result in arbitrary dimension $d$.
\begin{Thm}[Local Lipschitz]
\label{Thm_l2_lower_bound_higherD}
Let $d\geq 2$ and $\hat{\mu}_1,\hat{\mu}_2\in \widehat{\mathcal{M}}^{N,d}_{c_{\min}}(2\frac{\sqrt{d}}{N})$ and assume that $\hat{\mu}_1,\hat{\mu}_2$ satisfy
\begin{align} \label{eq_error_small_higherD}
    \sum_{\genfrac{}{}{0pt}{}{k\in\Z^d}{\|k\|_2\leq N}} |\hat{\mu}_1(k)-\hat{\mu}_2(k)|^2 < \left(\frac{3}{2}\right)^{d-1} \frac{N^d c_{\min}^2}{d^{d/2}}.
\end{align}
Then, for every $t\in Y^{\hat{\mu}_1}$ there is exactly one $t'=\eta(t)\in Y^{\hat{\mu}_2}$ with $\|t-\eta(t)\|_{\T^d}\leq \frac{\sqrt{d}}{2N}$ and vice versa. Moreover, we have the following estimate:
\begin{align*}
    \sum_{\genfrac{}{}{0pt}{}{k\in\Z^d}{\|k\|_2\leq N}} |\hat{\mu}_1(k)-\hat{\mu}_2(k)|^2 \geq  \sum_{t\in Y_1} 10 \left(\frac{3}{2}\right)^{d-2} \frac{d-1}{d^{d/2} d^2} N^{d+2} c_{\min}^2 \|t-\eta(t)\|_{\T^d}^2 + \frac{2}{4^d d^{d/2}}  N^d  |c_t^{(1)}-c_{\eta(t)}^{(2)}|^2.
\end{align*}
\end{Thm}
Note that the univariate case $d=1$ is not included in this theorem because our approach needed a slightly larger separation than $2\frac{\sqrt{d}}{N}=\frac{2}{N}$ in \cref{Thm_Lipschitz_univariate}. For the proof, we use the same techniques as before and a tensor product approach by defining
\begin{align*}
    \psi(x)=\psi(x_1,\dots,x_d)=\left((2\pi N)^2 + \sum_{s=1}^d \frac{\partial^2}{\partial x_s^2}\right) \bigotimes_{\ell=1}^d (\varphi *\varphi)(x_\ell)
\end{align*}
as in \cref{Prop_psi_generell} and using 
\begin{align*}
     \varphi(x)=\begin{cases} \cos^2\left(\frac{\pi x}{q}\right), &\quad |x|<\frac{q}{2},\\ 0, &\quad \text{otherwise.} \end{cases}
\end{align*}
as before. Note that we directly restricted ourselves to $p=2$ even if the idea should also work for $p\neq 2$ by finding a suitable $\varphi$.
\begin{Lem} \label{Lem_psi_highD}
If $d\geq 2$, the function $\psi$ satisfies
\begin{align*}
    \psi(x)&=\left[(2\pi N)^2-\frac{4d\pi^2}{q^2}\right] \prod_{\ell=1}^d (\varphi*\varphi)(x_\ell) 
    +\frac{\pi^2}{q^2} \left[\sum_{s=1}^d \left(\frac{q}{2\pi} \sin\left(\frac{2\pi|x_s|}{q}\right)+q-|x_s|\right) \prod_{i\neq s} (\varphi*\varphi)(x_i) \right]
\end{align*}
and has its global maximum in 0 if and only if $Nq\geq \sqrt{d}$. We set $q=\frac{\sqrt{d}}{N}$ and find 
\begin{align}
    \psi(0)-\psi(x)&\geq \begin{cases}  \frac{10}{3}\left(d-1\right) \left(\frac{3q}{8}\right)^{d-1} \pi^2 \frac{\|x\|_{\infty}^2}{q^3}, &\quad \|x\|_{\infty}<\frac{q}{2}, \\
     \left[\frac{5d}{6}-\frac{1}{3}-\frac{4}{3}m(d-1)\right]\left(\frac{3q}{8}\right)^{d-1} \pi^2 \frac{\|x\|_{\infty}}{q^2} , &\quad \frac{q}{2}\leq \|x\|_{\infty} \leq aq,
    \end{cases} \nonumber\\
    &\geq \left(d-\frac{1}{2}\right) \left(\frac{3q}{8}\right)^{d-1} \pi^2 \frac{\|x\|_{\infty}^2}{q^3}, \quad 0\leq \|x\|_{\infty} \leq q \label{eq_low_bound_general_x}
\end{align}
where $a\in\left(\frac{1}{2},1\right]$ and $m$ as in \eqref{eq_Diff_Quotient}.
\end{Lem}
\begin{proof}
The representation of $\psi$ is a direct corollary of \cref{Lem_phi_explicit}. Additionally, the global maximality of $\psi$ in 0 for $Nq\geq \sqrt{d}$ can be proven with the same estimates as in the proof of \cref{Lem_phi_explicit} by noting that the condition $Nq\geq \sqrt{d}$ implies $(2\pi N)^2-\frac{4d\pi^2}{q^2}\geq 0$. In order to understand why the condition is also necessary, we rewrite
\begin{align*}
    \psi(x)=\sum_{s=1}^d h(x_s) \prod_{i\neq s} (\varphi*\varphi)(x_i)
\end{align*}
where $h: [-q,q] \to \R$,
\begin{align*}
    h(t)=4\pi^2\left(\frac{N^2}{d}-\frac{1}{q^2}\right)\left(\frac{q-|t|}{4} \left[1+\frac{1}{2}\cos\frac{2\pi t}{q}\right] + \frac{3}{8} \frac{q}{2\pi} \sin\frac{2\pi |t|}{q}\right)+\frac{\pi^2}{q^2} \left(\frac{q}{2\pi}\sin\frac{2\pi|t|}{q}+q-|t|\right).
\end{align*}
Expanding $h$ into a power series, one finds
\begin{align*}
    h(t)=\frac{3q\pi^2}{2}\left(\frac{N^2}{d}-\frac{1}{3q^2}\right) - \left(\frac{N^2}{d}-\frac{1}{q^2}\right)\frac{q\pi^2}{4} \left(\frac{2\pi t}{q}\right)^2+\mathcal{O}(|t|^3)
\end{align*}
and therefore we can derive that $h$ has a local minimum in 0 if $Nq<\sqrt{d}$. Thus, there is $t^*\in[-q,q]$ such that $h(0)<h(t^*)$ implying $\psi(t^*\cdot \mathbf{e}_1)>\psi(0)$, i.e. $\psi$ does not have its global maximum in 0.\par
After the estimate
\begin{align*}
    \psi(0)-\psi(x)&\geq \psi(0) -\frac{\pi^2}{q^2} \left[q(\varphi*\varphi)^{d-1}(0)+(d-1)q(\varphi*\varphi)^{d-2}(0)(\varphi*\varphi)(\|x\|_{\infty})\right]
\end{align*}
for $\|x\|_{\infty}\leq \frac{q}{2}$ or
\begin{align*}
    \psi(0)-\psi(x)&\geq \psi(0) -\frac{\pi^2}{q^2} \left[\frac{q}{2}(\varphi*\varphi)^{d-1}(0)+(d-1)q(\varphi*\varphi)^{d-2}(0)(\varphi*\varphi)(\|x\|_{\infty})\right]
\end{align*}
for $\frac{q}{2}\leq \|x\|_{\infty}\leq aq$, we can apply \cref{Lem_bound_conv} in order to obtain the proposed bounds. The last inequality for the general case of $0\leq \|x\|_{\infty} \leq q$ can be deduced by setting $a=1$ and hence $m=-\frac{1}{8}$.
\end{proof}

\begin{proof}[Proof of \cref{Thm_l2_lower_bound_higherD}]
We set $q=\frac{\sqrt{d}}{N}$. Analogously to the proof of \cref{Thm_l2_lower_bound_2D}, condition \eqref{eq_error_small_higherD} is chosen such that $Y_3=\emptyset$. As in the proof of \cref{Lem_sequence_of_bounds}, we first bound $\psi(0)-\psi(x)$ for $\|x\|_{\infty}\in \left[\frac{q}{2},q\right]$ in order to shrink the maximally possible distance $\|t^*-\eta(t^*)\|_{\T^d}$. Proceeding iteratively, we end up with $\|t^*-\eta(t^*)\|_{\T^d}\leq a_k q$ where $a_0=1$,
\begin{align*}
    a_{k+1}=\begin{cases}
    \frac{3d}{5d-2-8m_k(d-1)}, &\quad a_k>\frac{1}{2}, \\
    \frac{1}{2}, &\quad \text{else,}
    \end{cases} \quad \text{and} \quad m_k=\frac{(\varphi*\varphi)(a_kq)-(\varphi*\varphi)(q/2)}{a_kq-q/2}.
\end{align*}
Again, one can show $a_k\to \frac{1}{2}$ and this proves the first claim. The second one follows by applying \cref{Lem_psi_highD} and using  
\begin{align*}
    \psi(|t-\eta(t)|_{\T^d})\geq \psi\left(\frac{q}{2},\frac{q}{2},\dots,\frac{q}{2}\right) = \frac{8d\pi^2}{q^2} \left(\frac{q}{16}\right)^d
\end{align*}
as well as $\hat{\psi}(0)=4\pi^2 N^2 (q/2)^{2d}$.
\end{proof}

\subsection{Global results}
Next, we want to formulate global results which show actually that the Prony-map is Lipschitz continuous meaning that we have to fix explicit metrics on the input and output space of the Prony map. While the $\ell^2$-norm on the space of moments $\C^\mathcal{B}$ is natural from our previous local results, a metric on the parameter space is not completely obvious. But since we want to compare differences in the node set and the weights simultaneously, we defined the output of the Prony-map as a probability-like measure. The following well-known metric on the space of probability measures developed in the field of \textit{optimal transport} (e.g.\,cf.\,\cite{Villani_09}) can be extended to complex probability-like measures:\footnote{Even if \cite[p.\,98]{Peyre_19} mentions that the Wasserstein distance can be extended to signed measures with equal mass through the Kantorovich-Rubenstein representation and \cite[eq.\,(43)]{Lev_16} generalises the set of transport plans $\pi$ to complex measures, it is not entirely clear if a generalisation of the Wasserstein distance to complex measures $\mu_1,\mu_2$ has been considered before.}
\begin{Def}[Wasserstein distance for probability-like measures] \label{Def_Wasserstein}
Let $\mu_1,\mu_2\in\mathcal{M}(\T^d)$ for any $d\in\N_+$. We denote the set of complex measures on $\T^d\times \T^d$ with marginals $\mu_1$ and $\mu_2$ by $\Pi(\mu_1,\mu_2)$ and call $\pi\in\Pi(\mu_1,\mu_2)$ a \emph{transport plan}.\footnote{This means that $\pi\in \Pi(\mu_1,\mu_2)$ satisfies $\pi(A\times \T^d)=\mu_1(A)$ and $\pi(\T^d\times B)=\mu_2(B)$ for all measurable sets $A,B\in \T^d$.} Moreover, we define for some Lipschitz continuous function $f:\T^d \to \R$ its optimal Lipschitz constant as 
\begin{align*}
    \Lip(f)=\sup_{x\neq y} \frac{|f(x)-f(y)|}{\|x-y\|_{\T^d}}.
\end{align*}
The \emph{1-Wasserstein distance} between $\mu_1$ and $\mu_2$ is defined via taking absolute value in the dual formulation of the usual Wasserstein distance. In other words, we set
\begin{align*}
    W_1(\mu_1,\mu_2)= \sup_{f: \Lip(f)\leq 1} \left|\int_{\T^d} f(x) d(\mu_1-\mu_2)(x)\right|. 
\end{align*}
\end{Def}
In the following lemma, we summarise that this definition generalises the Wasserstein distance to a well-defined metric on $\mathcal{M}(\T^d)$.
\begin{Lem} \label{Lem_Wasserstein_Metrik}
For $\mu_1,\mu_2\in \mathcal{M}(\T^d)$ we have $W_1(\mu_1,\mu_2)<\infty$ and in particular
\begin{align*}
    W_1(\mu_1,\mu_2)\leq \frac{1}{\sqrt{2}} |\mu_1-\mu_2|(\T^d)
\end{align*}
where 
\begin{align*}
    |\mu|(X)=\sup\left\{\sum_{k=1}^{\infty} |\mu(X_k)|, \quad \bigcup_{k\in\N} X_k=X,\,X_k\text{ disjoint}\right\}
\end{align*}
is the \textit{total variation} of $\mu\in\mathcal{M}(X)$ for any topological space $X$ (cf.\,\cite[Chapter 9]{Axler_20}). Moreover, the Wasserstein distance defines a metric on $\mathcal{M}(\T^d)$ induced by the norm
\begin{align*}
    \|\mu\|_{\Lip^*}=\sup_{f: \Lip(f)\leq 1, \|f\|_{\infty}\leq 1} \left|\int_{\T^d} f(x) d\mu(x)\right|. 
\end{align*}
\end{Lem}
\begin{proof}
See appendix.
\end{proof}
Considering this metric for the parameter space has two main advantages. First, it links information about nodes and weights in one term. Secondly, the Wasserstein distance allows to compare parameter sets whose cardinality is not necessarily equal.
\begin{Thm}[Global Lipschitz] \label{Cor_global_Lipschitz}
Let $d\geq 2$. For $\hat{\mu}_1,\hat{\mu}_2\in \widehat{\mathcal{M}}^{N,d}_{c_{\min}}\left(\frac{2\sqrt{d}}{N}\right)$ the Prony-map $\mathscr{P}$ satisfies
\begin{align*}
    W_1(\mathscr{P}(\hat{\mu}_1),\mathscr{P}(\hat{\mu}_2)) \leq \sqrt{3M} \left(1+\frac{3}{\sqrt{2}}\right)  \left(\frac{2}{3}\right)^{d/2-1/2} \frac{d^{d/4}}{N^{d/2}} \|\hat{\mu}_1-\hat{\mu}_2\|_2
    \le 2.3 \|\hat{\mu}_1-\hat{\mu}_2\|_2
\end{align*}
for all $\hat{\mu}_1,\hat{\mu}_2 \in \widehat{\mathcal{M}}^{N,d}_{c_{\min}}\left(\frac{2\sqrt{d}}{N}\right)$ if we restrict the Prony-map to node sets with cardinality at most $M$.
\end{Thm}
We have to emphasize that this result for $\mathscr{P}$ is completely independent of $c_{\min}$. We use the following lemma for the proof:
\begin{Lem} \label{Lem_Wasser_estimate}
For two discrete probability-like measures $\mu_1=\sum_{t\in Y} c_t^{(1)} \delta_t$, $\mu_2=\sum_{t\in Y} c_{\eta(t)}^{(2)} \delta_{\eta(t)}$ with the same cardinality of the node set, we have
\begin{align*}
    W_1\left(\sum_{t\in Y} c_t^{(1)} \delta_t,\sum_{t\in Y} c_{\eta(t)}^{(2)} \delta_{\eta(t)}\right)\leq \sqrt{|Y|} \left[\left(\sum_{t\in Y} \|\mathbf{c}_t\|_2^2 \|t-\eta(t)\|_{\T^d}^2\right)^{1/2} \hspace{-0.3cm}+\frac{1}{\sqrt{2}} \left(\sum_{t\in Y} |c_t^{(1)}-c_{\eta(t)}^{(2)}|^2\right)^{1/2}\right]
\end{align*}
where $\mathbf{c}_t= \left(c_t^{(1)} , -c_{\eta(t)}^{(2)} \right)^T$ as in the proof of \cref{Thm_l2_lower_bound_2D}.
\end{Lem}
\begin{proof}
See appendix.
\end{proof}
\begin{proof}[Proof of \cref{Cor_global_Lipschitz}]
We use the notation of the proof of \cref{Thm_l2_lower_bound_2D}. The idea of the proof is to split the problem into an estimation for the subset $Y_1$ of nodes in $Y^{\hat{\mu}_1}$ such that there exists a unique close neighbour in $Y^{\hat{\mu}_2}$ on one hand and an estimation for nodes in the remaining node set $Y_3$ on the other hand. Abbreviating $C^{(1)}=\sum_{t\in Y_1} c_t^{(1)}$ and $C^{(2)}=\sum_{t\in Y_1} c_{\eta(t)}^{(2)}$ in the computation
\begin{align*}
    \mu_1-\mu_2&=\sum_{t\in Y_3} \tilde{c}_t \delta_t - \sum_{t\in Y_1} \left(\frac{C^{(2)}-C^{(1)}}{|Y_1|}\right) \delta_t + \sum_{t\in Y_1} \left(\frac{C^{(2)}-C^{(1)}}{|Y_1|}+c_t^{(1)}\right) \delta_t - c_{\eta(t)}^{(2)} \delta_{\eta(t)} \\
    &=\left(C^{(2)}-C^{(1)}\right)\hspace{-0.14cm} \left[\sum_{t\in Y_3} \frac{\tilde{c}_t\delta_t }{C^{(2)}-C^{(1)}} - \sum_{t\in Y_1} \frac{\delta_t}{|Y_1|} \right] \hspace{-0.1cm} + C^{(2)} \hspace{-0.12cm}\left[ \sum_{t\in Y_1} \frac{\frac{C^{(2)}-C^{(1)}}{|Y_1|}+c_t^{(1)}}{C^{(2)}} \delta_t - \frac{c_{\eta(t)}^{(2)}\delta_{\eta(t)}}{C^{(2)}} \right]\hspace{-0.1cm},
\end{align*}
we can shift mass such that we can write $\mu_1-\mu_2$ as a linear combination of two differences of complex measures which are again probability-like. Therefore, we can apply the triangle inequality of $\|\cdot\|_{\Lip^*}$ in order to divide the problem into
\begin{align*}
    W_1(\mu_1,\mu_2)&=\sup_{f:\Lip(f)\leq 1} \left|\int_{\T^d} f(x) d(\mu_1-\mu_2)(x)\right| \\
    &\leq \left|C^{(2)}-C^{(1)}\right| W_1\left(\sum_{t\in Y_3} \frac{\tilde{c}_t\delta_t }{C^{(2)}-C^{(1)}},\sum_{t\in Y_1} \frac{\delta_t}{|Y_1|}\right) \\
    &\quad+\left|C^{(2)}\right| W_1\left(\sum_{t\in Y_1} \frac{\frac{C^{(2)}-C^{(1)}}{|Y_1|}+c_t^{(1)}}{C^{(2)}} \delta_t,\sum_{t\in Y_1} \frac{c_{\eta(t)}^{(2)}\delta_{\eta(t)}}{C^{(2)}}\right).
\end{align*}
The first term can be handled with \cref{Lem_Wasserstein_Metrik} and we obtain
\begin{align*}
    &\quad \left|C^{(2)}-C^{(1)}\right| W_1\left(\sum_{t\in Y_3} \frac{\tilde{c}_t\delta_t }{C^{(2)}-C^{(1)}},\sum_{t\in Y_1} \frac{\delta_t}{|Y_1|}\right) \\
    &\leq \left|C^{(2)}-C^{(1)}\right| \frac{1}{\sqrt{2}} \left(\sum_{t\in Y_3} \frac{|\tilde{c}_t|}{|C^{(2)}-C^{(1)}|} + \sum_{t\in Y_1} \frac{1}{|Y_1|} \right) \\
    &\leq \frac{1}{\sqrt{2}} \left(\sqrt{M} \left(\sum_{t\in Y_3} |\tilde{c}_t|^2\right)^{1/2} + \left|C^{(2)}-C^{(1)}\right|\right).
\end{align*}
On the contrary, \cref{Lem_Wasser_estimate} and $|z+w|^2\leq 2|z|^2+2|w|^2$ for $z,w\in\C$ yield for the second term
\begin{align*}
    &\left|C^{(2)}\right| W_1\left(\sum_{t\in Y_1} \frac{\frac{C^{(2)}-C^{(1)}}{|Y_1|}+c_t^{(1)}}{C^{(2)}} \delta_t,\sum_{t\in Y_1} \frac{c_{\eta(t)}^{(2)}\delta_{\eta(t)}}{C^{(2)}}\right) \\
    &\leq \sqrt{|Y_1|} \left(\sum_{t\in Y_1} \left(\left|\frac{C^{(2)}-C^{(1)}}{|Y_1|}+c_t^{(1)}\right|^2+|c_{\eta(t)}^{(2)}|^2\right) \|t-\eta(t)\|_{\T^d}^2\right)^{1/2} \\
    &\quad +\frac{\sqrt{|Y_1|}}{\sqrt{2}} \left(\sum_{t\in Y_1} \left|\frac{C^{(2)}-C^{(1)}}{|Y_1|}+c_t^{(1)}-c_{\eta(t)}^{(2)}\right|^2\right)^{1/2} \\
    &\leq \sqrt{2|Y_1|} \left(\sum_{t\in Y_1} \|\mathbf{c}_t\|_2^2 \|t-\eta(t)\|_{\T^d}^2\right)^{1/2} \hspace{-0.15cm} + \sqrt{|Y_1|} \left(\sum_{t\in Y_1} |c_t^{(1)}-c_{\eta(t)}^{(2)}|^2\right)^{1/2} + \left(1+\frac{1}{\sqrt{2}}\right) |C^{(2)}-C^{(1)}|.
\end{align*}
As we can observe 
\begin{align*}
\left|C^{(2)}-C^{(1)}\right|=\left|\sum_{t\in Y_3} \tilde{c}_t\right|\leq \sqrt{M} \left(\sum_{t\in Y_3} \left|\tilde{c}_t\right|^2\right)^{1/2}
\end{align*} 
and $(a+b+c)^2\leq 3a^2+3b^2+3c^2$ for $a,b,c\in\R$, we end up with
\begin{align} \label{eq_upper_bound_Wass_squared}
    W_1(\mu_1,\mu_2)^2 \leq 3M \left[\left(1+\frac{3}{\sqrt{2}}\right)^2 \sum_{t\in Y_3} |\tilde{c}_t|^2 + 2 \sum_{t\in Y_1} \|\mathbf{c}_t\|_2^2 \|t-\eta(t)\|_{\T^d}^2 + \sum_{t\in Y_1} |c_t^{(1)}-c_{\eta(t)}^{(2)}|^2\right].
\end{align}
Reviewing the proof of \cref{Thm_l2_lower_bound_2D} and using \cref{Lem_psi_highD}, we bound the difference in the moments from below by
\begin{align*}
    \|\hat{\mu}_1-\hat{\mu}_2\|_2^2&\geq \frac{1}{2\hat{\psi}(0)} \left[2\psi(0) \sum_{t\in Y_3} |\tilde{c}_t|^2 + \sum_{t\in Y_1} \frac{1}{2} \left(\psi(0)-\psi\left(|t-\eta(t)|_{\T^d}\right)\right) \left|c_t^{(1)}+c_{\eta(t)}^{(2)}\right|^2 \right. \\
    &\quad + \frac{1}{2} \left(\psi(0)+\psi\left(|t-\eta(t)|_{\T^d}\right)\right) \left|c_t^{(1)}-c_{\eta(t)}^{(2)}\right|^2 + \left(\psi(0)-\psi\left(|t-\eta(t)|_{\T^d}\right)\right) \|\mathbf{c}_t\|_2^2 \\
    &\quad \left.\vphantom{\sum_{t\in Y}} + \psi\left(|t-\eta(t)|_{\T^d}\right) \left|c_t^{(1)}-c_{\eta(t)}^{(2)}\right|^2 \right] \\
    &\geq \frac{1}{2\hat{\psi}(0)} \left[2\psi(0) \sum_{t\in Y_3} |\tilde{c}_t|^2 + \sum_{t\in Y_1} \left(\vphantom{\sum}\psi(0)-\psi\left(|t-\eta(t)|_{\T^d}\right)\right) \|\mathbf{c}_t\|_2^2 + \frac{1}{2} \psi(0) \left|c_t^{(1)}-c_{\eta(t)}^{(2)}\right|^2 \right] \\
    &=\frac{\left(\frac{3}{2}\right)^{d} N^d}{3d^{d/2}} \left[2 \sum_{t\in Y_3} |\tilde{c}_t|^2 + \frac{d-\frac{1}{2}}{dq^2} \sum_{t\in Y_1} \|t-\eta(t)\|_{\T^d}^2 \|\mathbf{c}_t\|_2^2 + \frac{1}{2} \left|c_t^{(1)}-c_{\eta(t)}^{(2)}\right|^2 \right] \\
    &\geq \frac{\left(\frac{3}{2}\right)^{d} N^d}{3d^{d/2}} \min\left(\frac{2}{\left(1+\frac{3}{\sqrt{2}}\right)^2},\frac{3}{8},\frac{1}{2}\right) \frac{W_1(\mu_1,\mu_2)^2}{3M}
\end{align*}
where we applied \eqref{eq_upper_bound_Wass_squared} and $\frac{d-\frac{1}{2}}{dq^2}\geq \frac{3}{4}$ in the last inequality. Noting $2\left(1+\frac{3}{\sqrt{2}}\right)^{-2}\approx 0.205$ yields the proposed estimate which is finally simplified by noting that $M^{-1/d}\ge q \ge 2\sqrt{d}/N$.
\end{proof}
Finally, we remark that one can also obtain global results for $d=1$ using \cref{Thm_Lipschitz_univariate} and estimates of the Wasserstein metric from the proof of \cref{Cor_global_Lipschitz}. Moreover, the proof of \cref{Cor_global_Lipschitz} shows that the constants in front of $|c_{t}^{(1)}-c_{\eta(t)}^{(2)}|^2$ in \cref{Thm_Lipschitz_univariate,Thm_l2_lower_bound_2D,Thm_l2_lower_bound_higherD} can be improved at the cost of slightly worse constants in the term with $\|t-\eta(t)\|_{\T^d}^2$. 

\subsection{Lower bounds for singular values of Vandermonde matrices} \label{sec_Vandermonde}
As already observed by Diederichs in \cite{Diederichs_19} for $d=1$, the results from the previous subsection yield an estimate for the smallest singular value of the Vandermonde matrix
\begin{align} \label{eq_Vandermonde_multi}
    \mathscr{A}=\left(\eim{tk}\right)_{k\in \{k\in\Z^d: \|k\|_2\leq N\}, t\in Y}
\end{align}
in the situation of a clustered node set $Y$ where the maximal cluster size is two. The reason for this is that $\hat{\mu}_1,\hat{\mu}_2\in \widehat{\mathcal{M}}^{N,d}_{c_{\min}}(2\frac{\sqrt{d}}{N})$ result in a joint node set $Y=Y^{\hat{\mu}_1}\cup Y^{\hat{\mu}_2}$ where at most two nodes have distance smaller than $q:=\frac{\sqrt{d}}{N}$ in the $\|\cdot\|_{\T^d}$-norm, i.e. the cluster size is bounded by two.\footnote{In contrast to \cite{Nagel_20}, we say that nodes form a cluster if they are contained in a cube of side length $\frac{\sqrt{d}}{N}$. The rest of the notation is similar to \cite{Nagel_20}.} We call a cluster $\Lambda\subset Y$ consisting of two nodes \textit{pair cluster} and prove that the \textit{cluster separation} 
\begin{align*}
\dist(\Lambda_i,\Lambda_j)=\min(\|y_1-y_2\|_{\T^d}, y_1\in \Lambda_1, y_2\in \Lambda_2)
\end{align*}
satisfies $\dist(\Lambda_1,\Lambda_2)\geq q$ for any two clusters $\Lambda_1,\Lambda_2\subset Y$. We see this by fixing $y_1\in \Lambda_1\cap Y^{\hat{\mu}_1}$(the case $y_1\in \Lambda_1\cap Y^{\hat{\mu}_2}$ works analogously). If the cluster $\Lambda_1$ only consists of $y_1$, i.e.\,$\Lambda_1=\{y_1\}$, one has directly $\|y_1-y_2\|_{\T^d}\geq q$ for any $y_2\in Y$ by the definition of a cluster. Otherwise, we have $\Lambda_1=\{y_1,y_3\}$ and distinguish for $y_2\in \Lambda_2$ between the cases
\begin{itemize}
    \item $y_2\in Y^{\hat{\mu}_1}$: One has $\|y_1-y_2\|_{\T^d}\geq 2q>q$ by the separation of $ Y^{\hat{\mu}_1}$.
    
    \item $y_2\in Y^{\hat{\mu}_2}$: Using the notation of the nearest neighbour from the previous section, we take $y_3=\eta(y_1)\in Y^{\hat{\mu}_2}$ and observe
    \begin{align*}
        \|y_1-y_2\|_{\T^d}&\geq \|y_2-y_3\|_{\T^d} - \|y_3-y_1\|_{\T^d} \\
        &\geq 2q-q = q.
    \end{align*}
\end{itemize}
Therefore, $\hat{\mu}_1,\hat{\mu}_2\in \widehat{\mathcal{M}}^d_{c_{\min}}(2\frac{\sqrt{d}}{N})$ imply that the \textit{minimal cluster separation} 
\begin{align*}
    \Delta:=\min_{i,j} \dist(\Lambda_i,\Lambda_j)
\end{align*}
satisfies $\Delta\geq q=\frac{\sqrt{d}}{N}$. 
\begin{figure}[h]
    \centering
    \begin{tikzpicture}[>=triangle 45,x=6cm,y=6cm,scale=0.95]
    \draw[->] (0,0) --node[below]{$x_1$} (1.1,0);
    \draw[->] (0,0) node[below]{0} -- node[left]{$x_2$} (0,1.1);
    \draw[-] (1,-2pt) node[below]{1} -- (1,2pt);
    \draw[-] (-2pt,1) node[left]{1} -- (2pt,1);
    \foreach \x in {(0.14,0.26),(0.5,0.45),(0.8,0.2),(0.13,0.8),(0.18,0.22),(0.55,0.47),(0.8,0.17),(0.33,0.77),(0.9,0.6)}
    \draw[mark=*,mark size=2pt,mark options={color=blue}] plot coordinates {\x};
    \draw[|-|] (0.48,0.38) -- node[below]{$\frac{\sqrt{2}}{N}$} (0.58,0.38);
    \foreach \x in {(0.1,0.2),(0.48,0.41),(0.07,0.76),(0.74,0.13),(0.3,0.73),(0.88,0.56)}
    \draw [draw=black] \x rectangle ++(0.1,0.1);
    \draw[|-|] (0.13,0.7) -- node[below]{$\Delta=2\frac{\sqrt{2}}{N}$} (.33,.7);
    \draw[-] (0.8,0.2) -- (.8,.17);
    \draw[->] (0.9,0.15) node[right]{$\tau$} -- (0.81,0.18);
    \draw[->] (0.2,0.1) node[below]{$\Lambda_1$} -- (0.15,.2); 
    \end{tikzpicture}
    \caption{Visualisation of pair-cluster configuration in $d=2$ where nodes $t\in Y$ within the same box of side length $\frac{\sqrt{d}}{N}$ form a cluster and points from different clusters are at least $\Delta=2\frac{\sqrt{d}}{N}$ away from each other.}
    \label{fig:vis_of_clusters}
\end{figure}
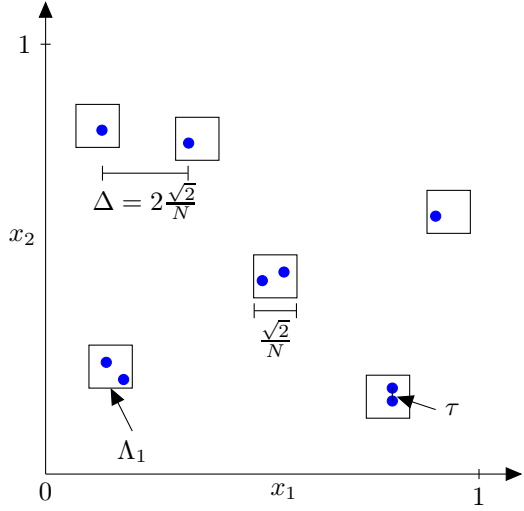
The converse holds with a minor modification by a factor of two: If a node set $Y$ with maximal cluster size two satisfies $\Delta\geq 2q=2\frac{\sqrt{d}}{N}$, one can define a partition of the node set $Y=Y_{\hat{\mu}_1}\cup Y_{\hat{\mu}_2}$ such that $\sep\,Y_{\hat{\mu}_j}\geq 2q$, $j=1,2$. Hence, the condition $\Delta\geq 2\frac{\sqrt{d}}{N}$ is sufficient for the application of \cref{Thm_l2_lower_bound_higherD}.
\begin{Cor}[Pair clustering, our result] \label{Cor_pairs_me}
Assume $d\geq 2$. Let $Y$ be a node set with at most pairwise clustering points, $\sep\,Y\geq \tau$ for some $\tau\in (0,\frac{\sqrt{d}}{N})$ and 
\begin{align*}
\Delta\geq 2q=2\frac{\sqrt{d}}{N}.
\end{align*} 
Then, a lower bound for the smallest non-zero singular value of the corresponding Vandermonde matrix $\mathscr{A}$ in \eqref{eq_Vandermonde_multi} is given by
\begin{align*}
    \sigma_{\min}(\mathscr{A}) \geq \sqrt{\frac{d-\frac{1}{2}}{3d^2}} \left(\frac{3}{2}\right)^{d/2}\frac{(N\tau) N^{d/2}}{ 2^{-1/2} d^{d/4}}.
\end{align*}
\end{Cor}
\begin{proof}
We know that we can find a suitable partition of the node set $Y=Y_{\hat{\mu}_1}\cup Y_{\hat{\mu}_2}$ and fix some arbitrary $c\in\C^{|Y|}$. Based on this, one can define exponential sums
\begin{align*}
    \hat{\mu}_j(k)=\sum_{t\in Y_{\hat{\mu}_j}} (-1)^{j+1} c_t \eim{tk}, \quad j=1,2,
\end{align*}
with $\hat{\mu}_1(k)-\hat{\mu}_2(k)=\left(\mathscr{A}c\right)_k$.
We follow the lines of the proof of \cref{Thm_l2_lower_bound_higherD} and compute
\begin{align*}
    \|\mathscr{A}c\|_2^2 &= 
    \sum_{\|k\|_2\leq N} |\hat{\mu}_1(k)-\hat{\mu}_2(k)|^2 \\
    &\geq \frac{\psi(0)}{\hat{\psi}(0)} \sum_{t\in Y_3} |c_t|^2 + \sum_{t\in Y_1} \frac{\psi(0)-\psi(|t-\eta(t)|_{\T^d})}{\hat{\psi}(0)} \left(|c_t|^2+|c_{\eta(t)}|^2\right) \\
    &\geq \left(\frac{3}{2}\right)^{d-1} \tau^2 N^{d+2} \frac{d-\frac{1}{2}}{d^{d/2}d^2} \|c\|_2^2
\end{align*}
where we used \eqref{eq_low_bound_general_x} for the last inequality. Using the relation $\sigma_{\min}(\mathscr{A})=\min_{\|c\|_2=1} \|\mathscr{A}c\|_2$, we derive the proposed lower bound. 
\end{proof}
We compare this to a similar result by Nagel \cite[Cor.\,3.4.15]{Nagel_20} which we modified by a frequency shift to fit into our setting.
\begin{Thm}[Pair clustering, Nagel's result] \label{Prop_Nagel_pairs}
Let $Y$ be a node set with at most pairwise clustering points, $\sep\,Y\geq \tau$ for some $\tau>0$ and 
\begin{align*}
\Delta\geq \frac{6d}{N} \left(\frac{2}{\tau N}\right)^{\frac{1}{d+1}}.
\end{align*} 
Then we have that the smallest non-zero singular value of the corresponding Vandermonde matrix $\mathscr{A}=\left(\eim{tk}\right)_{k\in \{-N,\dots,N\}^d, t\in Y}$ can be bounded from below by
\begin{align*}
    \sigma_{\min}(\mathscr{A}) \geq \frac{1}{6} \frac{(N\tau) N^{d/2}}{2^{-1/2} d^{d/4}}.
\end{align*}
\end{Thm}
While contrasting these results for the smallest singular value, we have to remark the following:
\begin{itemize}
    \item [(i)] The definition of a cluster is different in the two settings. Whereas in Nagel's language nodes form a cluster if they are contained in a cube of side length $\frac{1}{N}$ (independently of the dimension), we considered a side length of $\frac{\sqrt{d}}{N}$ in order to apply our theory from the previous section.
    \item [(ii)] The condition on the cluster separation is weaker in \cref{Cor_pairs_me} than in \cref{Prop_Nagel_pairs}. Especially, we emphasize that our lower bound for $\Delta$ is independent of $\tau$, i.e.\,$\Delta$ does not need to be adjusted for an arbitrarily small $\tau$.
    \item [(iii)] Our result for the smallest singular value has a exponentially better dimension-dependent constant.
\end{itemize}
Consequently, we were able to prove improved lower bounds for the smallest singular values of Vandermonde matrices in the special case of pairwise clustering nodes. Unfortunately, our method does not provide reasonable results for larger clusters containing $\lambda>2$ nodes in a cube of side length $\frac{\sqrt{d}}{N}$. As in the proof of \cref{Thm_l2_lower_bound_2D}, one would have to bound the eigenvalues of $3\times 3$-matrices $\mathbf{A}_y$ from below in this situation. This is much more difficult than for $2\times 2$-matrices and appears to be impossible for our choice of $\psi$.

\section{Conclusion}
Through the Lipschitz property, we are able to control the difference between two measures by their trigonometric moments. Apart from the intended application for the analysis of deep neural networks in microscopy, the result might be helpful in order to understand how close parametric methods like ESPRIT, performing a mapping from noisy measurements of the moments to an approximation of the measure, are to an \glqq optimal\grqq\,realisation of the Prony-map. Moreover, the setting of \cref{Thm_l2_lower_bound_higherD} directly corresponds to lower bounds for the smallest singular value of a Vandermonde matrix with pairwise clustering nodes. By this, we end up with lower bounds which have the same order in $N$ and the minimal separation of the nodes as the state of the art results, but our results has much weaker assumptions on the separation between different clusters.

\begin{Acknowledgements}
The authors gratefully acknowledge support by the DFG within the Collaborative Research Center 944 
``Physiology and dynamics of cellular microcompartments’’ and by the Volkswagen Foundation project ``Stability of Moment Problems and Super-Resolution Imaging’’.
\end{Acknowledgements}

\bibliography{Literatur}
\section*{Appendix} \label{Appendix}
We complete our presentation by adding some technical proofs.
\begin{proof}[Proof of \cref{Thm_Lipschitz_univariate}]
The proof works as the proof of \cref{Thm_l2_lower_bound_2D} with the localising function
\begin{align*}
    \psi(x)=\left((2\pi N)^2+\frac{\partial^2}{\partial x^2}\right) (\varphi*\varphi)(x)
\end{align*}
where we choose $\varphi$ as in \eqref{eq_cos_squared} with $q=\frac{\kappa}{N}$. Using \cref{Lem_phi_explicit} and the estimates for $\varphi*\varphi$ from  \cref{Lem_bound_conv}, one directly finds
\begin{align} \label{eq_bound_diff_1D}
    \psi(0)-\psi(x)&\geq 
    \begin{cases}
    5\pi^2 \frac{\kappa^2-1}{q^3} |x|^2,&\quad |x|\in \left[0,\frac{q}{2}\right], \\
    \frac{15}{4}\pi^2 \frac{\kappa^2-1}{q^2} |x|,&\quad |x|\in \left[\frac{q}{2},q\right],
    \end{cases} \\
    &\geq \frac{15}{4}\pi^2 \frac{\kappa^2-1}{q^3} |x|^2 \quad \text{for any $|x|\leq q$.} \nonumber
\end{align}
 Together with $\psi(0)=\frac{q\pi^2}{2}\frac{3\kappa^2-1}{q^2}$ and $\hat{\psi}(0)=(2\pi N)^2(q/2)^2$ we obtain the result for general $\kappa$. The optimal $\kappa$ for the lower bound can be found simply by maximising the expression $\left(1-\kappa^{-2}\right) \kappa^{-3}$ over all $\kappa>1$. For $\kappa=1$ one would find $\psi(0)-\psi(x)\geq C |x|^3$ yielding the worse order $|t-\eta(t)|^3$ as already observed in \cite{Diederichs_19}.\par
 Moreover, applying \eqref{eq_bound_diff_1D} and \eqref{eq_cond_univ} enable to find
 \begin{align*}
     \frac{3\kappa^2-1}{2\kappa^3} N c_{\min}^2 > \|\hat{\mu}_1-\hat{\mu}_2\|_2^2 \geq 15\frac{\kappa^2-1}{2q^2\kappa^2} c_{\min}^2 \|t^*-\eta(t^*)\|_{\T}.
 \end{align*}
 for any $t^*\in Y^{\hat{\mu}_1}$ with $\|t^*-\eta(t^*)\|_{\T}\in\left[\frac{q}{2},q\right]$. From here, we can conclude that $\|t^*-\eta(t^*)\|_{\T}\leq \frac{q}{2}$ if $\kappa^2\geq \frac{13}{9}$. Therefore, all $\kappa$ with $\kappa\geq \sqrt{\frac{13}{9}}$ including the optimal $\kappa=\sqrt{\frac{5}{3}}$, allow to improve the estimate by using only the part of \eqref{eq_bound_diff_1D} for $|x|\in\left[0,\frac{q}{2}\right]$. In the course of this, one can bound $\psi(|t-\eta(t)|_{\T})\geq \psi\left(\frac{q}{2}\right)=\frac{\pi^2}{4q} (\kappa^2+1)$ for the term in front of $\sum_{t\in Y^{\hat{\mu}_1}} |c_t^{(1)}-c_{\eta(t)}^{(2)}|^2$.
\end{proof}
\begin{proof}[Proof of \cref{Lem_sequence_of_bounds}]
Assume that there is some $t^*\in Y_{1}$ such that $\|t^*-\eta(t^*)\|_{\T^2}\in\left(\frac{q}{2},q\right]$. Setting $a_0=1$ and $m_0$ according to \cref{Lem_bound_conv}, we obtain by \cref{Lem_lower_bound_for_difference}, assumption \eqref{eq_error_small} and \eqref{eq_eigendecomposition}
\begin{align*}
    \hat{\psi}(0) \frac{3}{4} N^2 c_{\min}^2 > \hat{\psi}(0) \sum_{\genfrac{}{}{0pt}{}{k\in\Z^2}{\|k\|_2\leq N}} |\hat{\mu}_1(k)-\hat{\mu}_2(k)|^2 \geq (1-m_0) \frac{\pi^2}{q} c_{\min}^2\|t^*-\eta(t^*)\|_{\infty} 
\end{align*}
and this leads to $\|t^*-\eta(t^*)\|_{\T^2}\in\left(\frac{q}{2},a_1 q\right]$ where we defined
$a_{1}=\frac{3}{4-4m_0}$. We can repeat this procedure iteratively and generate a sequence of upper bounds $\|t^*-\eta(t^*)\|_{\T^2}\leq a_kq$ with
\begin{align} \label{eq_define_sequence}
    a_{k+1}=\frac{3}{4-4m_k}
\end{align}
for $k\geq 0$, if we can guarantee that $a_k\in\left[\frac{1}{2},1\right]$ implies $a_{k+1}\in\left[\frac{1}{2},1\right]$. Due to the convexity of $\varphi*\varphi$ on $[q/2,q]$ we can use $(\varphi*\varphi)'\left(\frac{q}{2}\right)=-\frac{1}{2}\leq (\varphi*\varphi)'(t) \leq (\varphi*\varphi)'(q)=0$ and the mean value theorem in order to find
\begin{align*}
    a_{k+1}=\frac{3}{4-4(\varphi*\varphi)'(c_k)} 
    \begin{cases}
    \leq \frac{3}{4-4\cdot 0} < 1 \\
    \geq \frac{3}{4-4(-\frac{1}{2})}=\frac{1}{2},
    \end{cases} 
\end{align*}
where $c_k \in\left(\frac{1}{2},a_k\right)$. So the sequence $(a_k)_{k\in\N}$ is well-defined and bounded. Moreover, monotonicity can be proven by defining the auxiliary function
\begin{align*}
    \gamma(t)=-\frac{4}{3} t^2 + \frac{19}{12}t -\frac{1}{2} + \frac{4}{3} \frac{t}{q} (\varphi*\varphi)(qt)
\end{align*}
and rewriting
\begin{align*}
    a_{k+1}-a_k=\frac{3}{4} \frac{\gamma(a_k)}{a_k-7/16-(\varphi*\varphi)(a_kq)/q}.
\end{align*}
One can easily prove that the denominator of the right hand side is positive for $a_k\in\left(\frac{1}{2},1\right]$, whereas the nominator is nonpositive for those $a_k$. Hence, the nonincreasing and bounded sequence converges to some $a\in\left[\frac{1}{2},1\right]$ and by taking the limit $k\to\infty$ on both sides of \eqref{eq_define_sequence} we can validate $a=\frac{1}{2}$ which is the intended contradiction. 
\end{proof}

\begin{proof}[Proof of \cref{Lem_Wasserstein_Metrik}]
At first, we prove the upper bound of the Wasserstein distance by the total variation distance. We separate the real and imaginary part of the measure and find
\begin{align*}
    W_1(\mu_1,\mu_2)&= \sup_{f: \Lip(f)\leq 1} \left|\int_{\T^d} f(x) d(\mu_1-\mu_2)(x)\right| \\
    &\leq \sup_{f: \Lip(f)\leq 1} \left|\int_{\T^d} f(x) d(\Re\mu_1-\Re\mu_2)(x)\right| + \sup_{f: \Lip(f)\leq 1} \left|\int_{\T^d} f(x) d(\Im \mu_1-\Im\mu_2)(x)\right| \\
    &\leq \inf_{\pi\in\Pi(\Re\mu_1,\Re\mu_2)} \int_{\T^d\times \T^d} \|x-y\|_{\T^d} d|\pi|(x,y) + \inf_{\pi\in\Pi(\Im\mu_1,\Im\mu_2)} \int_{\T^d\times \T^d} \|x-y\|_{\T^d} d|\pi|(x,y).
\end{align*}
For a signed measure $\mu$ on measurable set $(X,\Sigma)$ its \textit{Jordan decomposition} is
\begin{align*}
    \mu=\mu_+-\mu_-, \quad \mu_+(A)=\sup_{B\in \Sigma,B\subset A} \mu(B),\quad \mu_-(A)=-\inf_{B\in \Sigma,B\subset A} \mu(B),
\end{align*}
such that $\mu_+$ and $\mu_-$ are nonnegative measures on $(X,\Sigma)$ and $|\mu|=\mu_++\mu_-$ (e.g.\,cf.\,\cite{Axler_20}). Motivated by the proof of \cite[Thm.\,6.15]{Villani_09}, we bound the first infimum by choosing
\begin{align} \label{eq_choosing_pi}
    \pi_1^+=\left(\Id,\Id\right)_{\#}(\Re \mu_1-(\Re\mu_1-\Re\mu_2)_+) + \frac{1}{(\Re\mu_1-\Re\mu_2)_+(\T^d)} (\Re\mu_1-\Re\mu_2)_+\otimes (\Re\mu_1-\Re\mu_2)_- 
\end{align}
where we use $\Id$ for the identity map and the notation of the \textit{push-forward measure} $T_{\#}\mu$ defined by $T_{\#}\mu(A)=\mu(T^{-1}(A))$ for any measurable set $A$ and a measurable function $T$ (e.g.\,cf.\,\cite{Villani_09}).\footnote{Note that the case $(\Re\mu_1-\Re\mu_2)_+(\T^d)=0$ is trivial.} One can directly show that $\pi_1^+$ is admissible as a transport plan, i.e.\,$\pi_1^+\in\Pi(\Re\mu_1,\Re\mu_2)$. Hence, we find for arbitrary fixed $x_0\in\T^d$ by the triangle inequality
\begin{align*}
    &\quad \inf_{\pi\in\Pi(\Re\mu_1,\Re\mu_2)}\int_{\left(\T^d\right)^2} \|x-y\|_{\T^d} d|\pi|(x,y) \\
    &\leq \int_{\left(\T^d\right)^2}  \frac{\|x-y\|_{\T^d}\,d(\Re\mu_1-\Re\mu_2)_+(x) d(\Re\mu_1-\Re\mu_2)_-(y)} {(\Re\mu_1-\Re\mu_2)_+(\T^d)} \\
    &\leq \int_{\T^d} \|x-x_0\|_{\T^d} d(\Re\mu_1-\Re\mu_2)_+(x) + \int_{\T^d} \|y-x_0\|_{\T^d} d(\Re\mu_1-\Re\mu_2)_-(y) \\
    &\leq \frac{1}{2} \int_{\T^d} d(\Re\mu_1-\Re\mu_2)_+(x) + d(\Re\mu_1-\Re\mu_2)_-(x) \\
    &= \frac{1}{2}|\Re\mu_1-\Re\mu_2|(\T^d).
\end{align*}
The term for the imaginary part works analogously and the simple inequality $|\Re z| + |\Im z|\leq \sqrt{2}|z|$ for any $z\in\C$ gives the proposed estimate of $W_1(\mu_1,\mu_2)$ in terms of the total variation. As a direct consequence, one has finiteness of the Wasserstein distance since the total variation of a complex measure is finite (cf.\,\cite[Result\,9.17]{Axler_20}).\par
Nonnegativity, homogeneity, finiteness and the triangle inequality are trivial for $\|\cdot\|_{\Lip^*}$ and the condition $\|f\|_\infty\leq 1$ can be neglected if we consider $\|\mu-\nu\|_{\Lip^*}$ for probability-like $\mu,\nu$ leading to the definition of the Wasserstein metric from \cref{Def_Wasserstein}. In order to show that $\|\cdot\|_{\Lip^*}$ is definite, let $\int_{\T^d} f(x) d\mu(x)=0$ for all Lipschitz continuous functions $f: \T^d\to \R$ and assume that there is a measurable set $A\subset \T^d$ with $\mu(A)\neq 0$. Since $\mu$ is a Borel measure, we can consider $A$ being closed. Denoting the projection operator on $A$ by $\proj_A(x):=\argmin_{y\in A} \|x-y\|_{\T^d}$, we define for any $\epsilon>0$ a function $f_{\epsilon}:\T^d\to \R$,
\begin{align*}
    f_\epsilon(x)=\begin{cases}
    1,&\quad x\in A, \\ 1-\frac{1}{\epsilon}\|\proj_A(x)-x\|_{\T^d},&\quad x\in ((A+B_{\epsilon}(0))\setminus A) \cap \T^d, \\ 0,&\quad \text{else.}
    \end{cases}
\end{align*}
It is straightforward to show that $\Lip(f_{\epsilon})\leq \epsilon^{-1}$ and therefore $0=\int_{\T^d} f_{\epsilon} d\mu$ leading to the contradiction $\mu(A)=0$ by applying the dominated convergence theorem.
\end{proof}

\begin{proof}[Proof of \cref{Lem_Wasser_estimate}]
We separate the difference in the nodes and in the weights by inserting $\tilde{\mu}=\sum_{t\in Y} c_t^{(1)} \delta_{\eta(t)}$ in the computation
\begin{align*}
    W_1(\mu_1,\mu_2)&\leq W_1(\mu_1,\tilde{\mu}) + W_1(\tilde{\mu},\mu_2) \\
    &\leq \sup_{f: \Lip(f)\leq 1} \left|\sum_{t\in Y} c_t^{(1)} [f(t)-f(\eta(t))]\right| + \frac{1}{\sqrt{2}} |\tilde{\mu}-\mu_2|(\T^d)  \\
    &\leq \sum_{t\in Y} \left|c_t^{(1)}\right| \|t-\eta(t)\|_{\T^d} + \sum_{t\in Y} |c_t^{(1)}-c_t^{(2)}| \\
    &\leq \sqrt{|Y|} \left(\sum_{t\in Y} \left\|\mathbf{c}_t\right\|_2^2 \|t-\eta(t)\|_{\T^d}^2  + \sum_{t\in Y} |c_t^{(1)}-c_t^{(2)}|^2\right)
\end{align*}
where we used \cref{Lem_Wasserstein_Metrik} for the second inequality.
\end{proof}

\end{document}